\documentclass[a4paper]{amsart}
\usepackage[latin9]{inputenc}
\usepackage{mathtools}
\usepackage{mathrsfs}
\usepackage{amsthm}
\usepackage{amstext}
\usepackage{amssymb}
\usepackage[colorlinks, linkcolor=blue, anchorcolor=blue, citecolor=blue]{hyperref}

\makeatletter

\numberwithin{equation}{section}
\numberwithin{figure}{section}
\usepackage{enumitem}		
\theoremstyle{plain}
\newtheorem{thm}{\protect\theoremname}[section]
  \theoremstyle{definition}
  \newtheorem{defn}[thm]{\protect\definitionname}
  \theoremstyle{plain}
  \newtheorem{prop}[thm]{\protect\propositionname}
  \theoremstyle{remark}
  \newtheorem{rem}[thm]{\protect\remarkname}
  \theoremstyle{plain}
  \newtheorem{lem}[thm]{\protect\lemmaname}
  \theoremstyle{remark}
  \newtheorem{claim}[thm]{\protect\claimname}
  \theoremstyle{plain}
  \newtheorem{cor}[thm]{\protect\corollaryname}
  \theoremstyle{definition}
  \newtheorem{example}[thm]{\protect\examplename}

\usepackage{a4wide}
\usepackage{amsfonts}

\newcommand{\set}[1]{\left\{#1\right\}}
\newcommand{\qg}{{\mathbb G}}

\renewcommand{\a}{\alpha}
\newcommand{\f}{\varphi}

\newcommand{\nm}[1]{\|#1\|}

\newcommand{\irr}{\mathrm{Irr}(\mathbb{G})}

\DeclareMathOperator{\tr}{Tr}

\makeatother

  \providecommand{\claimname}{Claim}
  \providecommand{\corollaryname}{Corollary}
  \providecommand{\definitionname}{Definition}
  \providecommand{\examplename}{Example}
  \providecommand{\lemmaname}{Lemma}
  \providecommand{\propositionname}{Proposition}
  \providecommand{\remarkname}{Remark}
\providecommand{\theoremname}{Theorem}

\begin{document}

\title{$L_{p}$-improving convolution operators on finite quantum groups}

\author{Simeng Wang}

\address{Laboratoire de Mathématiques, Université de Franche-Comté, 25030
Besançon Cedex, France and Institute of Mathematics, Polish Academy
of Sciences, ul. \'{S}niadeckich 8, 00-956 Warszawa, Poland}

\email{simeng.wang@univ-fcomte.fr}

\subjclass[2010]{Primary: 20G42, 46L89. Secondary: 43A22, 46L30, 46L51.}

\keywords{$L_{p}$-improving operators, compact quantum groups, positive convolution
operators.}
\begin{abstract}
We characterize positive convolution operators on a finite
quantum group $\mathbb{G}$ which are $L_{p}$-improving. More precisely,
we prove that the convolution operator $T_{\varphi}:x\mapsto\varphi\star x$
given by a state $\varphi$ on $C(\mathbb{G})$ satisfies 
\[
\exists1<p<2,\quad\|T_{\varphi}:L_{p}(\mathbb{G})\to L_{2}(\mathbb{G})\|=1
\]
if and only if the Fourier series $\hat{\varphi}$ satisfy $\|\hat{\varphi}(\alpha)\|<1$
for all nontrivial irreducible unitary representations $\alpha$, and if
and only if the state $(\varphi\circ S)\star\varphi$ is non-degenerate
(where $S$ is the antipode). We also prove that these $L_{p}$-improving
properties are stable under taking free products, which gives a method
to construct $L_{p}$-improving multipliers on infinite compact quantum
groups. Our methods for non-degenerate states yield a general formula for computing idempotent 
states associated to Hopf images, which generalizes earlier work of Banica, 
Franz and Skalski. 
\end{abstract}
\maketitle

\section*{Introduction}

The convolution operators or multipliers constitute a central part of Fourier analysis. One among phenomena studied on the
circle group $\mathbb{T}$ is the existence and behavior of positive
Borel measures that convolve $L_{p}(\mathbb{T})$ into $L_{q}(\mathbb{T})$
with finite $q>p$ for a given $1<p<\infty$, which are considered
to be $L_{p}$-improving measures. An example due to Oberlin
\cite{oberlin1982convolution} is the Cantor-Lebesgue measure
supported by the usual middle-third Cantor set. Oberlin revealed that,
after a careful analysis on the structure of this measure, this result
can be reduced to proving that there exists $p<2$ such that 
\[
\|\mu\star f\|_{2}\leq\|f\|_{p},\quad f\in L_{p}(\mathbb{Z}/3\mathbb{Z})
\]
where the $L_{p}$-norms are those taken with respect to the normalized
counting measure on the cyclic group $\mathbb{Z}/3\mathbb{Z}=\{0,1,2\}$
with three elements and $\mu$ is the probability measure with mass $1/2$ at
$0$ and at $2$. Motivated by these results, Ritter
showed in \cite{ritter1984convolution} that, if $G$ is an arbitrary
finite group and $T_{\mu}:f\mapsto\mu\star f$ is the convolution
operator associated to a probability measure $\mu$ on $G$, then
\[
\left(\exists\, p<2,\ \|T_{\mu}:L_{p}(G)\to L_{2}(G)\|=1\right)\ \Leftrightarrow\ G=\langle ij^{-1}:i,j\in\mathrm{supp}\,\mu\rangle,
\]
which provides a more general method to construct $L_{p}$-improving
measures on groups.

In this paper we give an alternative approach related to these topics,
in the context of quantum groups and noncommutative $L_{p}$-spaces.
We show that, for a finite quantum group $\mathbb{G}$ and a state
$\varphi$ on $C(\qg)$, denoting by $\hat{\varphi}$ the Fourier
series of $\varphi$ and writing $\psi=(\f\circ S)\star\f$, $S$
being the antipode, the following assertions are equivalent (Theorem
\ref{cor:convolution_quantul}): 
\begin{enumerate}[label=(\arabic{enumi})]
\item there exists $1<p<2$ such that, 
\[
\forall\ x\in C(\qg),\ \nm{\f\star x}_{2}\leq\nm{x}_{p}\ ;
\]

\item $\nm{\hat{\f}(\a)}<1$ for all $\a\in\irr\setminus\{1\}$ ; 
\item For any nonzero $x\in C(\qg)_{+}$, there exists $n\geq1$ such that
$\psi^{\star n}(x)>0$. 
\end{enumerate}
The last assertion should be interpreted
as claiming that the ``support'' of $\varphi$ ``generates'' the
quantum group $\mathbb{G}$, which will be explained in the last section.
We will illustrate by example in Remark \ref{rem:finiteness is crucial}
that the finiteness condition in the above conclusion is rather crucial
and cannot be removed.

In particular, the result characterizes the Fourier-Schur multipliers
on finite groups which have an $L_{p}$-improving property.
Let $\Gamma$ be a finite group and $\varphi$ be a positive definite
function on $\Gamma$. Let $M_{\varphi}$ be the associated Fourier-Schur
multiplier operator determined by $M_{\varphi}(\lambda(\gamma))=\varphi(\gamma)\lambda(\gamma)$
for all $\gamma\in\Gamma$. Then 
\[
\exists1<p<2,\quad\|M_{\varphi}x\|_{2}\leq\|x\|_{p},\quad x\in C^{*}(\Gamma)
\]
if and only if $|\varphi(\gamma)|<1$ for any $\gamma\in\Gamma\setminus\{e\}$.

We should emphasize that our argument relies essentially on new
and interesting properties on the unital trace preserving operators
on noncommutative $L_{p}$-spaces, based on the recent work of Ricard and Xu \cite{ricardxu2014convexLp}. In fact, the following
fact proved in Theorem \ref{Lp improvement_spectral gap_strong version}
plays a key role in our argument. For a finite dimensional
C{*}-algebra $A$ equipped with a faithful tracial state $\tau$, and $T:A\to A$ a unital trace preserving map, the $L_{p}$-improving
property
\begin{equation}
\exists1<p<2,\quad\|T:L_{p}(A)\to L_{2}(A)\|=1\label{eq:Lp improving}
\end{equation}
holds if and only if we have the following ``spectral gap'':
\[
\sup_{\substack{x\in A\backslash\{0\},\tau(x)=0}
}\frac{\|Tx\|_{2}}{\|x\|_{2}}<1.
\]
We provide two proofs of this result, where one is based on very elementary arguments with an additional assumption of $2$-positivity and another, which is rather short, on \cite{ricardxu2014convexLp}. In Theorem \ref{thm:free prod Lp improving} we also show that the
$L_{p}$-improving property \eqref{eq:Lp improving} remains stable
under the free products. This method permits us to give $L_{p}$-improving
convolution operators for infinite
quantum groups.

In this paper we also include some simple properties of non-degenerate states on compact quantum groups with applications. We prove in Lemma \ref{lem:non degenerate} that the convolution Ces\`{a}ro limit of a non-degenerate state is the Haar state, which not only contributes to the proof of our main result, but also yields a generalization of \cite[Theorem 2.2]{banicafranzskalski2012inner} concerning the computation of idempotent states associated to Hopf images.

We end this introduction with a brief description of the organization
of the paper. Section \ref{sec:-improvement-and-spectral} deals with
the characterization of unital trace preserving $L_{p}$-improving
operators on finite dimensional C{*}-algebras and their free products. In Section \ref{sec:Fourier-analysis} we present some
preliminaries on compact quantum
groups and the related Fourier analysis. Here we give a short and explicit
calculation of Fourier series for compact quantum groups, parallel
to the case of classical compact groups, which does not exist in other
literature. In Section \ref{sec:convolution and hopf image} we obtain some properties of non-degenerate
states on a general compact quantum group.
The last Section \ref{sec:Some--improving-convolutions} is devoted
to the positive convolution operators on finite quantum groups, and
constructions of operators with similar properties on infinite compact
quantum groups by free product.

\section{$L_{p}$-improvement and spectral gaps\label{sec:-improvement-and-spectral}}

\subsection{Basic notions}
Let us firstly present some preliminaries and notations on noncommutative $L_p$-spaces and free products for later use. All the facts mentioned below are well-known.

\subsubsection{Noncommutative $L_{p}$-spaces}

Here we recall some basics of noncommutative $L_{p}$-spaces
on finite von Neumann algebras. We refer to \cite{takesaki2002opeI}
for the theory of von Neumann algebras and to \cite{pisierxu2003nclp}
for more information on noncommutative $L_{p}$-spaces. Let $\mathcal{M}$
be a finite von Neumann algebra equipped with a normal faithful tracial
state $\tau$. Let $1\leq p<\infty$. For each $x\in\mathcal{M}$,
we define
\[
\|x\|_{p}=\left[\tau(|x|^{p})\right]^{1/p}.
\]
One can show that $\|\|_{p}$ is a norm on $\mathcal{M}$. The completion
of $(\mathcal{M},\|\|_{p})$ is denoted by $L_{p}(\mathcal{M},\tau)$
or simply by $L_{p}(\mathcal{M})$. The elements of $L_{p}(\mathcal{M})$
can be described by densely defined closed operators measurable with
respect to $(\mathcal{M},\tau)$, as in the commutative case. For
convenience, we set $L_{\infty}(\mathcal{M})=\mathcal{M}$ equipped
with the operator norm. Since $|\tau(x)|\leq\|x\|_{1}$ for all $x\in\mathcal{M}$,
$\tau$ extends to a continuous functional on $L_{1}(\mathcal{M})$.
Let $1\leq p,q,r\leq\infty$ be such that $1/p+1/q=1/r$. If $x\in L_{p}(\mathcal{M})$
and $y\in L_{q}(\mathcal{M})$, then $xy\in L_{r}(\mathcal{M})$ and
the following Hölder inequality holds:
\[
\|xy\|_{r}\leq\|x\|_{p}\|y\|_{q}.
\]
In particular, if $r=1$, $|\tau(xy)|\leq\|xy\|_{1}\le\|x\|_{p}\|y\|_{q}$
for arbitrary $x\in L_{p}(\mathcal{M})$ and $y\in L_{q}(\mathcal{M})$.
This defines a natural duality between $L_{p}(\mathcal{M})$ and $L_{q}(\mathcal{M})$:
$\langle x,y\rangle=\tau(xy)$. For any $1\leq p<\infty$ we have
$L_{p}(\mathcal{M})^{*}=L_{q}(\mathcal{M})$ isometrically.

\subsubsection{Free products}

We firstly recall some constructions of free product of C{*}-algebras,
for which we refer to \cite{vdn1992freerandom} and \cite{nicaspeicher2006freeproba}
for details. Consider a family of unital C*-algebras
$(A_{i},\phi_{i})_{i\in I}$ with distinguished faithful states $\phi_{i}$
and associated GNS constructions $(\pi_{i},H_{i})$. Set $\mathring{A_{i}}=\ker\phi_{i}$ and $\mathring{a_{i}}=a_{i}-\phi_{i}(a_{i})1$
for each $i$ and $a_{i}\in A_{i}$. Construct a vector space
\begin{equation}
A=\mathbb{C}\mathbf{1}\oplus\bigoplus_{n\ge1}\Big(\bigoplus_{i_{1}\neq i_{2}\neq\cdots\neq i_{n}}\mathring{A}_{i_{1}}\otimes\mathring{A}_{i_{2}}\otimes\cdots\otimes\mathring{A}_{i_{n}}\Big).\label{algebraic free product}
\end{equation}

We equip $A$ with an algebra structure such that $\mathbf{1}$
is the identity and the multiplication of a letter $a\in\mathring{A_{i}}$
with an elementary tensor $a_{1}\otimes a_{2}\otimes\cdots\otimes a_{n}$
in $\mathring{A}_{i_{1}}\otimes\mathring{A}_{i_{2}}\otimes\cdots\otimes\mathring{A}_{i_{n}}$
is defined as 
\[
a\cdot(a_{1}\otimes a_{2}\otimes\cdots\otimes a_{n})=\left\{ \begin{array}{ll}
\ a\otimes a_{1}\otimes a_{2}\otimes\cdots\otimes a_{n}, & \textrm{if \ensuremath{i_{1}\neq i},}\\
\begin{array}{l}
\mathring{(aa_{1})}\otimes a_{1}\otimes a_{2}\otimes\cdots\otimes a_{n}\\
\quad+\phi_{i_{1}}(aa_{1})a_{2}\otimes\cdots\otimes a_{n},
\end{array} & \textrm{if \ensuremath{i_{1}=i}.}
\end{array}\right.
\]
Moreover, we give an involution on $A$ by 
\[
(a_{1}\otimes a_{2}\otimes\cdots\otimes a_{n})^{*}=a_{n}^{*}\otimes a_{2}^{*}\otimes\cdots\otimes a_{1}^{*}.
\]
In this sense $A$ becomes a $*$-algebra, and each $A_{i}$ can be
viewed as a $*$-subalgebra in $A$ by identifying $A_{i}$ with $\mathbb{C}\mathbf{1}\oplus\mathring{A}_{i}$
in the big direct sum. We call $A$ the \emph{algebraic free product}
of $(A_{i})_{i\in I}$.

It then can be shown that the algebra $A$ admits a faithful $*$-representation
$(\pi,H,\xi)$ such that $\pi|_{A_{i}}=\pi_{i}$ for each $i\in I$
and $\phi(\cdot)\coloneqq\langle\pi(\cdot)\xi,\xi\rangle$ restricted
on $A_{i}$ coincides with $\phi_{i}$. Moreover the state $\phi$
is faithful on $A$. Then the \emph{reduced C{*}-algebraic free product}
of $(A_{i})_{i\in I}$ is the C{*}-algebra generated by $\pi(A)$
in $B(H)$, i.e., the norm closure of $\pi(A)$ in $B(H)$, denoted
by $*_{i\in I}^{c_{0}}A_{i}$; and the state extends to $*_{i\in I}^{c_{0}}A_{i}$,
called the \emph{free product state} of $(\phi_{i})_{i\in I}$ and
denoted by $*_{i\in I}\phi_{i}$. If moreover each $A_{i}=\mathcal{M}_{i}$
is a von Neumann algebra and each $\phi_{i}$ is normal, then the
weak closure of $\pi(A)$ in $B(H)$, is defined to be the \emph{von
Neumann algebraic free product} of $(\mathcal{M}_{i})_{i\in I}$,
denoted by $\bar{*}_{i\in I}\mathcal{M}_{i}$, and the free product
state $\phi=*_{i\in I}\phi_{i}$ is also normal. Also, we remark that
if each $\phi_{i}$ is a tracial state, then $\phi=*_{i\in I}\phi_{i}$
is also tracial. 

Let $A_{i}$ and $B_{i}$ be unital C{*}-algebras with distinguished
faithful states $\phi_{i}$ and $\psi_{i}$ ($i\in I$) respectively,
and let $T_{i}:A_{i}\to B_{i}$ be a unital state preserving map for
each $i\in I$. Set $(A,\phi)=*_{i\in I}(A_{i},\phi_{i})$ and $(B,\psi)=*_{i\in I}(B_{i},\psi_{i})$.
Then it is obvious that 
\[
T(a_{1}a_{2}\cdots a_{n})=T_{i_1}(a_{1})\cdots T_{i_n}(a_{n})\qquad(a_{k}\in\mathring{A}_{i_{k}},\forall k,i_{1}\neq i_{2}\neq\cdots\neq i_{n})
\]
defines a unital state preserving map from the \emph{algebraic} free
products $(A,\phi)$ to $(B,\psi)$. We denote by $T=*_{i\in I}T_{i}$, and call
it the \emph{free product map} of the $T_{i}$. Similarly, we may define the \emph{c-free (conditionally free) product state} in the sense of Bo\.{z}ejko, Leinert and Speicher [BLS96]. Let $(A_{i},\phi_{i})$ be as above and let $\rho_i$ be further states respectively on $A_i$ for each $i$.  The conditional free product of $(\rho_i)_i$ is the functional $\omega:=*_{(\psi_i)} \rho_i$ on $(A,\phi)=*_{i\in I}(A_{i},\phi_{i})$  defined by the prescription $\omega(1) = 1$ and
\begin{equation*} \omega (a_1 \cdots a_n) = \rho_{i(1)} (a_1) \cdots \rho_{i(n)} (a_n)
\end{equation*}
for all $n \geq 1 $, $i(1)\neq \cdots \neq i(n)$ elements in $I$ and $a_j \in \ker \phi_{i(j)}$ for $j=1,\ldots,n$. It is shown in \cite[Theorem 2.2]{bozejkoleinertspeicher96condfree} that the conditional free product of states is again a state.

\subsection{$L_p$-improving operators}
Let $A$ be a finite dimensional C{*}-algebra equipped with a faithful
tracial state $\tau$. The associated noncommutative $L_{p}$-spaces will be denoted by $L_{p}(A)$. For a subset $E\subset A$,
we denote by $E_{+}$ the positive part of $E$. 

Recall that $A$ can be identified with a direct sum of matrix algebras,
that is, there exist some finite dimensional Hilbert spaces $H_{1},\ldots,H_{m}$
such that the following $*$-isomorphism holds 
\[
A\simeq B(H_{1})\oplus\cdots\oplus B(H_{m}).
\]
We will not distinguish the above two C{*}-algebras in the sequel.
For each $i\in\{1,\ldots,m\}$, let $\xi_{1}^{i},\ldots,\xi_{n_{i}}^{i}$
be an orthonormal basis for $H_{i}$, and define the operator $e_{pq}^{i}\in B(H_{i})$
by $e_{pq}^{i}(v)=\langle v,\xi_{q}^{i}\rangle_{H_{i}}\xi_{p}^{i}$
for all $v\in H_{i}$ and $p,q\in\{1,\ldots,n_{i}\}$. Take any $x=x_{1}\oplus\cdots\oplus x_{m}\in A$
with $x_{i}\in B(H_{i})$ for each $i\in\{1,\ldots,m\}$, and let
$\lambda_{1}^{i},\ldots,\lambda_{n_{i}}^{i}$ be the eigenvalues of
$|x_{i}|\in B(H_{i})$ ($1\leq i\leq m$) ranged in non-increasing
order and counted according to multiplicity. We can find a direct
sum of unitaries $u=u_{1}\oplus u_{2}\oplus\cdots\oplus u_{m}$ with
$u_{i}\in B(H_{i})$ for each $i$ such that $|x_{i}|(u_{i}\xi_{k}^{i})=\lambda_{k}^{i}(u_{i}\xi_{k}^{i})$
for all $k\in\{1,\ldots,n_{i}\}$ and $i\in\{1,\ldots,m\}$, that
is, $u^{*}|x|u=\sum_{i}\sum_{k=1}^{n_{i}}\lambda_{k}^{i}e_{kk}^{i}$.
If we write $\beta_{k}^{i}=\tau(e_{kk}^{i})\in[0,1]$ for $k\in\{1,\ldots,n_{i}\}$
and $i\in\{1,\ldots,m\}$, then the $L_{p}$-norm of $x$ for $1\leq p<\infty$
is
\begin{equation}
\|x\|_{p}^p=\tau(u^{*}|x|^{p}u)=\tau\left(\sum_{i=1}^{m}\sum_{k=1}^{n}(\lambda_{k}^{i})^{p}e_{kk}^{i}\right)=\sum_{i=1}^{m}\sum_{k=1}^{n}(\lambda_{k}^{i})^{p}\beta_{k}^{i}.\label{eq:L_p norm}
\end{equation}

We will prove in this section the result below. 
\begin{thm}
\label{Lp improvement_spectral gap}Let $A$ be a finite dimensional
C{*}-algebra equipped with a faithful tracial state $\tau$, and $T:A\to A$
be a unital $2$-positive trace preserving map on $A$. Then 
\[
\exists1\leq p<2,\quad\nm{Tx}_{2}\leq\nm{x}_{p},\ x\in A
\]
if and only if 
\begin{equation}
\sup_{\substack{x\in A\backslash\{0\},\tau(x)=0}
}\frac{\|Tx\|_{2}}{\|x\|_{2}}<1.\label{eq:spetral gap}
\end{equation}
\end{thm}
\begin{rem}
Equivalently we can rewrite the above condition \eqref{eq:spetral gap}
as 
\[
\sup_{\substack{x\in A\backslash\{0\},\tau(x)=0}
}\frac{\langle|T|x,x\rangle}{\|x\|_{2}^{2}}<1,
\]
which means exactly that the whole eigenspace of $|T|$ for the eigenvalue
$1$ is just $\mathbb{C}1$. In this sense we refer to the above inequality
as a spectral gap phenomenon of $T$. 
\end{rem}
Recall that the $L_{2}$-norms assert some differential properties. The following lemma is elementary.
\begin{lem}
Let $A$ be a C{*}-algebra with a state $\f$ and $T:A\to A$ be a
positive map on $A$. Let $O\subset A_{h}$ be an open set in
the space $A_{h}$ of all selfadjoint elements in $A$. The function
$f:O\ni x\mapsto\f((Tx)^{2})$ is infinitely (Fr\'{e}chet) differentiable in $O$
and for $x\in O$, $f'(x)=\f(TxT\cdot)+\f(T\cdot Tx)$, $f''\equiv2\f(T\cdot T\cdot)$,
$f^{(n)}\equiv0$, $n\geq3$. \end{lem}

In general a norm estimate can be reduced to the argument on positive
cones. 
\begin{lem}[{{\cite[Remark 9]{ricardxu2014convexLp}}}]
\label{lem:positive reduction}Let $\mathcal{M}$ be a von Neumann
algebra and $T:L_{p}(\mathcal{M})\to L_{q}(\mathcal{M})$ be a bounded
linear map for $1\leq p,q\leq\infty$. Assume that $T$ is $2$-positive
in the sense that $\mathrm{Id}_{\mathcal{\mathbb{M}}_{2}}\otimes T$
maps the positive cone of $L_{p}(\mathbb{M}_{2}\otimes\mathcal{M})$
to that of $L_{q}(\mathcal{\mathbb{M}}_{2}\otimes\mathcal{M})$. Then

\[
\|Tx\|_{q}\leq\|T(|x|)\|_{q}^{1/2}\|T(|x^{*}|)\|_{q}^{1/2},\quad x\in L_{p}(\mathcal{M}).
\]
Consequently, 
\[
\|T\|=\sup\{\|Tx\|_{q}:x\in L_{p}(\mathcal{M})_{+},\|x\|_{p}\leq1\}.
\]

\end{lem}
Now we give the proof of the theorem. 
\begin{proof}[Proof of the theorem]
Assume firstly $1\leq p<2$ and $\nm{Tx}_{2}\leq\nm x_{p}$ for all
$x\in A$. Note that $\|Tx\|_2=\||T|x\|_2$. Observe that $T^*$ is also a positive trace preserving map on $A$, and hence so is $|T|$. We choose an element $x\in A$ such that $\tau(x)=0$ and $|T|x=\lambda x$, with
$$\lambda=\sup_{\substack{x\in A\backslash\{0\},\tau(x)=0}}\frac{\|Tx\|_{2}}{\|x\|_{2}}=\sup_{\substack{x\in A\backslash\{0\},\tau(x)=0}}\frac{\||T|x\|_{2}}{\|x\|_{2}}.$$
Since $|T|$ is a positive map on $A$ and $\lambda\in \mathbb R$, we may assume that $x=x^*$. For any self-adjoint element $y\in A$ and $1\leq q<\infty$, it is easy to compute that
$$ \left.\frac{d^2}{d\varepsilon^2}\nm{1+\varepsilon y}_q\right|_{\varepsilon=0}=(q-1)\tau(y^2)>0.$$
Note also that by assumption
$$ \|1+\lambda \varepsilon x\|_2\leq \|1+\varepsilon x\|_p,\quad \varepsilon> 0.$$
Then taking the second derivative at $\varepsilon=0$ we get $\lambda^2\leq (p-1)<1$, as desired.

Now we suppose \eqref{eq:spetral gap} holds. Set $\mathring{A}=\{x\in A\bigm|\tau(x)=0\}$
and take $\sigma=\{x\in A_{+}\bigm|\tau(x)=1\}=(1+\mathring{A})_{+}$
which is exactly the set of positive elements in the unit sphere of $L_{1}(A)$.
We first show that there exists $1\leq p<2$ and a neighborhood $U$
of $1$ such that 
\begin{equation}
\forall x\in U\cap\sigma,\ \nm{Tx}_{2}\leq\nm x_{p}.\label{eq:Lp improvement,positive cone in unit L_1}
\end{equation}
To begin with, we consider 
\[
F(x)=\nm{Tx}_{2}-\nm x_{2},\ x\in A_{+}.
\]
Using the previous lemma we see that $F$ is infinitely differentiable
at any $x\in A_{+}\setminus\{0\}$ and 
\begin{align*}
F'(x) & (y)=\nm{Tx}_{2}^{-1}\tau((Tx)(Ty))-\nm{x}_{2}^{-1}\tau(xy),\quad y\in A\\
F''(x) & (y_{1},y_{2})=-\nm{Tx}_{2}^{3}\tau((Tx)(Ty_{1}))\tau((Tx)(Ty_{2}))+\nm{Tx}_{2}^{-1}\tau((Ty_{1})(Ty_{2}))\\
 & \,\quad+\nm{x}_{2}^{-3}\tau(xy_{1})\tau(xy_{2})-\nm{x}_{2}^{-1}\tau(y_{1}y_{2}),\quad y_{1},y_{2}\in A.
\end{align*}
Since $T$ is unital and preserves the trace, it follows that for
$y\in\mathring{A}$, 
\[
F'(1)(y)=0,\quad F''(1)(y,y)=\nm{Ty}_{2}^{2}-\nm y_{2}^{2}.
\]
Then consider the second order Taylor expansion of $F$ at $1$.
We can find a $\delta_{1}>0$ such that for all $\nm{y}_{2}\leq\delta_{1}$,
$y\in\mathring{A}$, we have $1+y\in A_{+}$ and 
\begin{align*}
F(1+y) & =F(1)+F'(1)(y)+\frac{1}{2}F''(1)(y,y)+R_{1}(y)\\
 & =\frac{1}{2}(\nm{Ty}_{2}^{2}-\nm y_{2}^{2})+R_{1}(y),\quad R_{1}(y)=o(\nm y_{2}^{2}).
\end{align*}
Recall that by \eqref{eq:spetral gap}, $\nm{Ty}_{2}^{2}-\nm y_{2}^{2}<0$
for $y\in\mathring{A}$. Thus by continuity, 
\[
c\coloneqq\sup\{\nm{Ty}_{2}^{2}-\nm y_{2}^{2}:y\in\mathring{A},\nm y_{2}=1\}<0.
\]
Since the function $y \mapsto \nm{Ty}_{2}^{2}-\nm y_{2}^{2} $ is $2$-homogeneous, we get 
\[
\forall y\in\mathring{A},\quad\nm{Ty}_{2}^{2}-\nm y_{2}^{2}\leq c\nm y_{2}^{2}.
\]
Take $\delta_{0}\in(0,\delta_{1})$ such that 
\[
\forall y\in\mathring{A},\nm y_{2}\leq\delta_{0},\quad\frac{|R_{1}(y)|}{\nm y_{2}^{2}}<\frac{|c|}{4}.
\]
Then for $y\in\mathring{A}$, $\nm y_{2}\leq\delta_{0},$ 
\[
F(1+y)=\frac{1}{2}(\nm{Ty}_{2}^{2}-\nm y_{2}^{2})+R_{1}(y)\leq\frac{c}{4}\nm y_{2}^{2}.\leqno(*)
\]
On the other hand, consider 
\[
G(x)=\nm x_{2}-\nm{x}_{p},\quad x=1+y,y=y^{*}\in A,\nm y_{2}<\delta_{0}.
\]
Let $y=y^{*}\in A$ with $\nm y_{2}<\delta_{0}$, then by \eqref{eq:L_p norm}
we may take some $K\in\mathbb{N}$ and $\beta_{1},\ldots,\beta_{K}\in[0,1]$
such that the $L_{p}$-norm of $x=1+y$ for $1\leq p<\infty$ is exactly
\[
\nm{1+y}_{p}=\left(\sum_{i=1}^{K}\beta_{i}(1+\lambda_{i})^{p}\right)^{\frac{1}{p}}\leqno(**)
\]
where $(\lambda_{i})_{i}\subset\mathbb{R}$ is the list of eigenvalues
of $y$. So in order to estimate $G$, we consider the function $g$
on $\mathbb{R}^{K}$ defined as 
\[
g(\xi)=\left(\sum_{i=1}^{K}\beta_{i}(1+\xi_{i})^{2}\right)^{\frac{1}{2}}-\left(\sum_{i=1}^{K}\beta_{i}(1+\xi_{i})^{p}\right)^{\frac{1}{p}},\quad\xi=(\xi_{1},\ldots,\xi_{K})\in\mathbb{R}^{K}.
\]
A straightforward calculation gives 
\[
\frac{\partial g}{\partial\xi_{i}}(0)=0,\ \frac{\partial^{2}g}{\partial\xi_{i}\partial\xi_{j}}(0)=(p-2)\beta_{i}\beta_{j},\ \frac{\partial^{2}g}{\partial\xi_{i}^{2}}(0)=(2-p)(\beta_{i}-\beta_{i}^{2}),\ 1\leq i\neq j\leq K.
\]
So by the Taylor formula 
\[
g(\xi)=\frac{1}{2}\sum_{i}(2-p)(\beta_{i}-\beta_{i}^{2})\xi_{i}^{2}+\frac{1}{2}\sum_{j\neq k}(p-2)\beta_{j}\beta_{k}\xi_{j}\xi_{k}+R_{2}(\xi),\ R_{2}(\xi)=o(\nm\xi^{2}).
\]
If $2-\frac{|c|}{8}\leq p\leq2$ and $0<\delta<\delta_{0}$ is such
that $|R_{2}(\xi)|\leq\frac{|c|}{8}\sum_{i=1}^{K}\beta_{i}\xi_{i}^{2}$
whenever $\sum_{i=1}^{K}\beta_{i}\xi_{i}^{2}\leq\delta^{2}$,
then for any $\xi\in\mathbb{R}^{K}$ with $\sum_{i=1}^{K}\beta_{i}\xi_{i}^{2}\leq\delta^{2}$,
\[
|g(\xi)|\leq\frac{1}{2}(2-p)\sum_{i=1}^{K}(\beta_{i}-\beta_{i}^{2})\xi_{i}^{2}+\frac{1}{2}(2-p)\sum_{i=1}^{K}\beta_{i}^{2}\xi_{i}^{2}+\frac{|c|}{8}\sum_{i=1}^{K}\beta_{i}\xi_{i}^{2}<\frac{|c|}{4}\sum_{i=1}^{K}\beta_{i}\xi_{i}^{2}.
\]
This, together with $(**)$, implies that, putting $\lambda=(\lambda_{1},\ldots,\lambda_{K})$,
\[
G(1+y)=g(\lambda)\leq\frac{|c|}{4}\sum_{i=1}^{K}\beta_{i}\lambda_{i}^{2}=\frac{|c|}{4}\nm y_{2}^{2},\quad\nm y_{2}\leq\delta.
\]
Combined with $(*)$ we deduce 
\[
\nm{Tx}_{2}-\nm x_{p}=F(1+y)+G(1+y)\leq0,\ x=1+y,\ y\in\mathring{A},\ \nm y_{2}\leq\delta,
\]
for all $p\geq2-\frac{|c|}{8}\coloneqq p_{1}$. So $U=\{1+y\bigm|y=y^{*}\in A,\|y\|_{2}<\delta\}$
is the desired neighborhood in \eqref{eq:Lp improvement,positive cone in unit L_1}.

Now we can derive the inequality for all $x\in\sigma$. For $x\in\sigma\setminus U\subset(1+\mathring{A})_{+}\setminus\{1\}$,
we write $x=1+y$ with $y\in\mathring{A}\backslash\{0\}$ and then
by \eqref{eq:spetral gap} and the trace preserving property we have
$\nm{Tx}_{2}^{2}=1+\|Ty\|_{2}^{2}<1+\|y\|_{2}^{2}=\|x\|_{2}^{2}$.
Note also that $\sigma$ is compact, so we can find $M<1$ such that
$\nm{Tx}_{2}/\nm x_{2}<M$ for all $x\in\sigma\backslash U$. Given
$p<2$, let $C_{p}$ be the optimal constant for the inequality $\nm x_{2}\leq C_{p}\nm x_{p}$
for $x\in A$, then $C_{p}\to1$ when $p\to2$. Take $p_{0}\geq p_{1}$
such that $C_{p_{0}}\leq M^{-1}$. We get then 
\[
\forall x\in\sigma\backslash U,\quad\frac{\nm{Tx}_{2}}{\nm x_{p}}\leq MC_{p}\leq1,\quad p_{0}\leq p\leq2.
\]
As a result, for all $p\in[p_{0},2]$, it holds that 
\[
\nm{Tx}_{2}\leq\nm x_{p},\quad x\in\sigma.
\]
Since the norm is homogeneous and $T$ is $2$-positive, the above
inequality holds for all $x\in A$ as well. 
\end{proof}
Apart from the above elementary proof, we would like to give an alternative
simpler approach which yields a little bit stronger conclusion. The argument,
however, depends heavily on the following recent and deep result
on the convexity of $L_{p}$-spaces:
\begin{thm}[{\cite[Theorem 1]{ricardxu2014convexLp}}]
\label{lem:ricard_xu_conv}Let $\mathcal{M}$ be a von Neumann algebra
equipped with a faithful semifinite normal trace $\phi$. Let $\mathcal{N}$
be a von Neumann subalgebra such that the restriction of $\phi$ to
$\mathcal{N}$ is semifinite. Denote by $\mathcal{E}$ the unique
$\phi$-preserving conditional expectation from $\mathcal{M}$ onto
$\mathcal{N}$. For $1<p\leq2$, we have
\[
\|x\|_{p}^{2}\geq\|\mathcal{E}x\|_{p}^{2}+(p-1)\|x-\mathcal{E}x\|_{p}^{2},\qquad x\in L_{p}(\mathcal{M}).
\]
For $2<p<\infty$, the inequality is reversed.
\end{thm}
Immediately we may deduce Theorem \ref{Lp improvement_spectral gap} as follows. Note that the result below is slightly stronger than the statement of Theorem \ref{Lp improvement_spectral gap}.
\begin{thm}
\label{Lp improvement_spectral gap_strong version}Let $A$ be a finite
dimensional C{*}-algebra equipped with a faithful tracial state $\tau$,
and $T:A\to A$ be a unital trace preserving map on $A$. Then 
\begin{equation}\label{Lp improv operator condition}
\exists 1<p<2,\ \forall\ x\in A,\ \|Tx\|_{2}\leq\|x\|_{p}
\end{equation}
if and only if $$\lambda\coloneqq\sup_{\substack{x\in A\backslash\{0\},\tau(x)=0}
}\frac{\|Tx\|_{2}}{\|x\|_{2}}<1.$$
Moreover, if the above assertions are satisfied, then $$\lambda\leq c_p^{-1}\sqrt{p-1},\ \text{ where }c_p=\sup_{x\in A\setminus\{0\},\tau(x)=0}\frac{\|x\|_2}{\|x\|_p}.$$
\end{thm}
\begin{proof}
The necessity has been already proved in the proof of Theorem 1.1. Now, assume $\lambda <1$. Let $x\in A$ and $y=x-\tau(x)1$. Write $a=\tau(x)$.
Since $T$ is trace preserving, $\tau(Ty)=\tau(y)=0$. For $p\leq2$ we denote by $c_{p}$ the best constant with $\|\cdot\|_{2}\leq c_{p}\|\cdot\|_{p}$.
Then $(p-1)/c_{p}^{2}\to1$ when $p\to2$. Take $p<2$ such that $(p-1)/c_{p}^{2}>\lambda^{2}$,
then we have
\begin{align*}
\|Tx\|_{2}^{2} & =\|a1+Ty\|_{2}^{2}=|a|^{2}+\|Ty\|_{2}^{2}\leq|a|^{2}+\lambda^{2}\|y\|_{2}^{2}\\
 & \leq|a|^{2}+\lambda^{2}c_{p}^{2}\|y\|_{p}^{2}\leq|a|^{2}+(p-1)\|y\|_{p}^{2}\leq\|x\|_{p}^{2},
\end{align*}
whence \eqref{Lp improv operator condition}.
\end{proof}
\begin{rem}
Let $A$ be a finite
dimensional C{*}-algebra equipped with a faithful tracial state $\tau$,
and $T:A\to A$ be a unital trace preserving map on $A$. Consider the restriction of $T$ on the subspace $\{x\in A:\tau(x)=0\} $ of $A$ and its adjoint, then we see that 
$$\sup_{\substack{x\in A\backslash\{0\},\tau(x)=0}
}\dfrac{\|Tx\|_{2}}{\|x\|_{2}}=\sup_{\substack{x\in A\backslash\{0\},\tau(x)=0}
}\dfrac{\|T^{*}x\|_{2}}{\|x\|_{2}}.$$
Then the above theorem also implies that if there exists $1<p<2$ such that 
\[
\forall\ x\in A,\ \|Tx\|_{2}\leq\|x\|_{p},
\]
then
\[
\forall\ x\in A,\ \|T^*x\|_{2}\leq\|x\|_{p},
\]
and equivalently for $2<q<\infty$ with $1/p+1/q=1$, 
\[
\forall\ x\in A,\ \|Tx\|_{q}\leq\|x\|_{2}.
\]
\end{rem}
\medskip

It is easy to see that the free product of unital trace preserving completely positive maps can be extended to the $L_p$-spaces on using the interpolation between $L_1$ and $L_\infty$. But in general it is a delicate problem for the extension of algebraic free product of unital trace preserving maps onto the associated $L_p$-spaces. Here we provide a method to construct unital trace
preserving $L_{p}$-improving operators on the free product of finite-dimensional
C{*}-algebras. To see this we need the following trivial claim.
\begin{claim}
Let $\mathcal{M}$ be a finite von Neumann algebra equipped with a
faithful tracial state $\tau$. If the vectors $e_{1},\ldots,e_{m}\in\mathcal{M}$
are orthonormal in $L_{2}(\mathcal{M},\tau)$ and denote $c=\max_{1\leq k\leq m}\|e_{k}\|_{\infty}^{2}$,
then for $\alpha_{1},\ldots,\alpha_{m}\in\mathbb{C}$ and $2\leq q\leq \infty$,
\[
\|\sum_{k=1}^{m}\alpha_{k}e_{k}\|_{q}\leq(cm)^{\frac{1}{2}-\frac{1}{q}}\|\sum_{k=1}^{m}\alpha_{k}e_{k}\|_{2}.
\]
\end{claim}
\begin{proof}
Note that 
\begin{align*}
\|\sum_{k=1}^{m}\alpha_{k}e_{k}\|_{\infty} & \leq c^{1/2}\sum_{k=1}^{m}|\alpha_{k}|\leq c^{1/2}m^{1/2}\left( \sum_{k=1}^{m}|\alpha_{k}|^{2} \right) ^{1/2},
\end{align*}
which gives the claim for $q=\infty$. The inequality for $2\leq q\leq \infty$
then follows from the Hölder inequality.\end{proof}
\begin{thm}
\label{thm:free prod Lp improving}Let $(A_{i},\tau_{i})$, $1\leq i\leq n$
be a finite family of finite dimensional C{*}-algebras and set $(\mathcal{A},\tau)=\bar{*}_{1\leq i\leq n}(A_{i},\tau_{i})$
to be the von Neumann algebraic free product. For each $1\leq i\leq n$,
$T_{i}$ is a unital trace preserving map such that 
\[
\|T_{i}:L_{p}(A_{i})\to L_{2}(A_{i})\|=1
\]
for some $1<p<2$. Then the (algebraic) free product map $T=*_{1\leq i\leq n}T_{i}$
on $*_{1\leq i\leq n}A_{i}$ extends to a map such that
\[
\|T:L_{p'}(\mathcal{A})\to L_{2}(\mathcal{A})\|=1
\]
for some $1<p'<2$. \end{thm}
\begin{proof}
By the previous theorem and remark, 
\begin{equation}
\lambda=\max_{1\leq i\leq n}\sup_{x\in \mathring A_i}
\dfrac{\|T_{i}x\|_{2}}{\|x\|_{2}}=\max_{1\leq i\leq n}\sup_{x\in \mathring A_i}
\dfrac{\|T_{i}^{*}x\|_{2}}{\|x\|_{2}}<1.\label{eq:s gap with adjoint}
\end{equation}
Consider $R=T^{*}$ and $R_{i}=T_{i}^{*}$ for all $1\leq i\leq n$,
then $R=R_{1}*\cdots*R_{n}$. By density, consider $x\in*_{1\leq i\leq n}(A_{i},\tau_{i})$
in the algebraic free product and we will show that 
\[
\|Rx\|_{q}\leq\|x\|_{2}
\]
for some $q>2$ independent of the choice of $x$. Now fix some $r\geq1$.
For each $i$, choose a family  $(e_{k}^{(i)})_{k=1}^{n_{i}}$ of eigenvectors of $|R_i|$ which forms an orthonormal
basis of $\mathring{A}_{i}$ under $\tau_{i}$, then $E_{r}=\{e_{\underline{k}}^{\underline{i}}=e_{k_{1}}^{(i_{1})}\cdots e_{k_{r}}^{(i_{r})}:1\leq k_{j}\leq n_{j},1\leq j\leq r,i_{1}\neq\cdots\neq i_{r}\}$
forms an orthonormal basis of $\oplus_{i_{1}\neq\cdots\neq i_{r}}\mathring{A}_{i_{1}}\otimes\cdots\otimes\mathring{A}_{i_{r}}$ which are also eigenvectors of $|R|$.
Note that $|E_{r}|\leq n^{r}m^{r}$ for $m=\max_{j}n_{j}$. Write
additionally $c=\max_{k,i}\|e_{k}^{(i)}\|_{\infty}^{2}$. Then for
any $y_{r}\in\oplus_{i_{1}\neq\cdots\neq i_{r}}\mathring{A}_{i_{1}}\otimes\cdots\otimes\mathring{A}_{i_{r}}$
the
above claim yields
\begin{equation}\label{Lq norm estimate of red words}
\|y_{r}\|_{q}\leq(cnm)^{r(\frac{1}{2}-\frac{1}{q})}\|y_{r}\|_{2}.
\end{equation}
Write $x=\tau(x)1+\sum_{r\geq1}x_{r}$ where $x_{r}\in\oplus_{i_{1}\neq\cdots\neq i_{r}}\mathring{A}_{i_{1}}\otimes\cdots\otimes\mathring{A}_{i_{r}}$.
Note that $\|Rx_{r}\|_{2}\leq\lambda^{r}\|x_{r}\|_{2}$ according
to \eqref{eq:s gap with adjoint} and the choice of $E_r$.
Together with Theorem \ref{lem:ricard_xu_conv} and \eqref{Lq norm estimate of red words},
\begin{align*}
\|Rx\|_{q}^{2} & \leq|\tau(x)|^{2}+(q-1)\|\sum_{r\geq1}Rx_{r}\|_{q}^{2}\leq|\tau(x)|^{2}+(q-1)\left(\sum_{r\geq1}\|Rx_{r}\|_{q}\right)^{2}\\
 & \leq|\tau(x)|^{2}+(q-1)\left(\sum_{r\geq1}(cnm)^{r(\frac{1}{2}-\frac{1}{q})}\|Rx_{r}\|_{2}\right)^{2}\\
 & \leq|\tau(x)|^{2}+(q-1)\left(\sum_{r\geq1}(cnm)^{r(\frac{1}{2}-\frac{1}{q})}\lambda^{r}\|x_{r}\|_{2}\right)^{2}.
\end{align*}
Observe that $(q-1)(cnm)^{\frac{1}{2}-\frac{1}{q}}$ tends to $1$ whenever
$q\to2$ and that $\lambda<1$, so we may choose $2<q<\infty$ such that
$\lambda(cnm)^{\frac{1}{2}-\frac{1}{q}}\leq(q-1)^{-1}$. For such a $q$ we
then have
\begin{align*}
\|Rx\|_{q}^{2} & \leq|\tau(x)|^{2}+(q-1)\left(\sum_{r\geq1}(q-1)^{-r}\|x_{r}\|_{2}\right)^{2}\\
 & \leq|\tau(x)|^{2}+(q-1)\sum_{k\geq1}(q-1)^{-2k}\sum_{r\geq1}\|x_{r}\|_{2}^{2}\\
 & <|\tau(x)|^{2}+\sum_{r\geq1}\|x_{r}\|_{2}^{2}=\|x\|_{2}^{2}.
\end{align*}
Take $1<p'<2$ such that $1/p'+1/q=1$. Then we get $\|T:L_{p'}(\mathcal{A})\to L_{2}(\mathcal{A})\|=1$.
\end{proof}

\section{\label{sec:Fourier-analysis}Preliminaries on quantum groups with Fourier analysis}

In this section we will do some preparations for discussing convolution operators in the quantum group framework. We will start with some preliminaries on compact quantum groups and then introduce the Fourier series in this setting.

\subsection{Compact quantum groups}

In this short paragraph we recall some basic definitions and properties of compact quantum
groups. All proofs of the facts mentioned below without references can
be found in \cite{woronowicz1998note} and \cite{maesvandaele1998note}.
\begin{defn}
Consider a unital C{*}-algebra $A$ and a unital $*$-homomorphism
$\Delta:A\to A\otimes A$ called \emph{comultiplication} on $A$ such that
$(\Delta\otimes\iota)\Delta=(\iota\otimes\Delta)\Delta$ and 
\[
\{\Delta(a)(1\otimes b):a,b\in A\}\quad\text{and}\quad\{\Delta(a)(b\otimes1):a,b\in A\} 
\]
are linearly dense in $A\otimes A$. Then $(A,\Delta)$ is called
a \emph{compact quantum group}. We denote $\mathbb{G}=(A,\Delta)$
and $A=C(\mathbb{G})$. We say that $\mathbb{G}$ is a \emph{finite}
quantum group if the space $A=C(\mathbb{G})$ is finite dimensional.\end{defn}

The following fact due to Woronowicz is fundamental in the quantum
group theory.
\begin{prop}
Let $\mathbb{G}$ be a compact quantum group. There exists
a unique state $h$ on $C(\mathbb{G})$ \emph{(}called the \emph{Haar
state} of $\mathbb{G}$\emph{)} such that for all $x\in C(\mathbb{G})$,
\[
(h\otimes\mathrm{\iota})\circ\Delta(x)=h(x)1=(\iota\otimes h)\circ\Delta(x).
\]

\end{prop}
Let $\mathbb{G}=(A,\Delta)$ be a compact quantum group and
consider an element $u\in A\otimes B(H)$ with $\dim H=n$. We identify
$A\otimes B(H)=\mathbb{M}_{n}(A)$, and write $u=[u_{ij}]_{i,j=1}^{n}$. Here,
$u$ is called an \emph{$n$-dimensional representation }of $\mathbb{G}$
if for all $j,k=1,...,n$, we have 
\begin{equation}
\Delta(u_{jk})=\sum_{p=1}^{n}u_{jp}\otimes u_{pk}.\label{eq:comultiplication}
\end{equation}
A representation $u$ is said to be\emph{ non-degenerate} if $u$
is invertible, \emph{unitary} if $u$ is unitary, and \emph{irreducible}
if the only matrices $T\in\mathbb{M}_{n}(\mathbb{C})$ with $Tu=uT$
are multiples of the identity matrix. Two representations $u,v\in\mathbb{M}_{n}(A)$
are called \emph{equivalent} if there exists an invertible matrix
$T\in\mathbb{M}_{n}(\mathbb{C})$ such that $Tu=vT$. Denote by $\mathrm{Irr}(\mathbb{G})$
the set of unitary equivalence classes of irreducible unitary representations
of $\mathbb{G}$. For each $\alpha\in\mathrm{Irr}(\mathbb{G})$, let
$u^{\alpha}\in C(\mathbb{G})\otimes B(H_{\alpha})$ be a representative
of the class $\alpha$ where $H_{\alpha}$ is the finite dimensional
Hilbert space on which $u^{\alpha}$ acts. 

With the notation above, the $*$-subalgebra $\mathcal{A}$ spanned by $\{u_{ij}^{\alpha}:u^{\alpha}=[u_{ij}^{\alpha}]_{i,j=1}^{n_{\alpha}},\alpha\in\mathrm{Irr}(\mathbb{G})\}$,
usually called the algebra of the polynomials on $\mathbb{G}$, is
dense in $C(\mathbb{G})$ , and the Haar state $h$ is faithful on
this dense algebra. In the sequel we denote $\mathcal{A}=\mathrm{Pol}(\mathbb{G})$.
Consider the GNS representation $(\pi_{h},H_{h})$ of $C(\mathbb{G})$,
then $\mathrm{Pol}(\mathbb{G})$ can be viewed as a subalgebra of
$B(H_{h})$. Define $C_{r}(\mathbb{G})$ (resp., $L_{\infty}(\mathbb{G})$)
to be the C{*}-algebra (resp., the von Neumann algebra) generated
by $\mathrm{Pol}(\mathbb{G})$ in $B(H_{h})$. Then $h$ extends to
a normal faithful state on $L_{\infty}(\mathbb{G})$.

It is known that there exists a linear antihomomorphism $S$ on $\mathrm{Pol}(\mathbb{G})$ such that \begin{equation}
S(S(a)^*)^*=a,\quad a\in \mathrm{Pol}(\mathbb{G}),\label{antipode}
\end{equation}
determined by 
$$S(u_{ij}^{\alpha})=(u_{ji}^\alpha)^*,\quad u^\alpha=[u_{ij}^{\alpha}]_{i,j=1}^{n_{\alpha}},\ \alpha\in\mathrm{Irr}(\mathbb{G}).$$
$S$ is called the \emph{antipode} of $\mathbb{G}$. For $a,b\in \mathrm{Pol}(\mathbb{G})$, we have
\begin{align}\label{antipode property}
S((\iota\otimes h)(\Delta(b)(1\otimes a)))&=(\iota\otimes h)((1\otimes b)\Delta(a)),\\
S((h \otimes\iota)((b\otimes 1)\Delta(a)))&=(h \otimes\iota)(\Delta(b)(a\otimes 1)).\nonumber
\end{align}

We will  use the \emph{Sweedler notation} for the comultiplication of an element $ a\in A $, i.e.
omit the summation and the index in the formula $ \Delta(a)=\sum_i a_{(1),i} \otimes a_{(2),i}  $ and write
simply $ \Delta(a)=\sum a_{(1)} \otimes a_{(2)}  $.

The Peter-Weyl theory for compact groups can be extended to the quantum
case. In particular, it is known that for each $\alpha\in\irr$ there
exists a positive invertible operator $Q_{\alpha}\in B(H_{\alpha})$
such that $\mathrm{Tr}(Q_{\alpha})=\mathrm{Tr}(Q_{\alpha}^{-1})\coloneqq d_{\alpha}$
and 
\begin{equation}
h(u_{ij}^{\alpha}(u_{lm}^{\beta})^{*})=\delta_{\alpha\beta}\delta_{il}\dfrac{(Q_{\alpha})_{mj}}{d_{\alpha}},\quad h((u_{ij}^{\alpha})^{*}u_{lm}^{\beta})=\delta_{\alpha\beta}\delta_{jm}\dfrac{(Q_{\alpha}^{-1})_{li}}{d_{\alpha}}\label{eq:haar state def}
\end{equation}
where $\beta\in\irr$, $1\leq i,j\leq\dim H_{\alpha}$, $1\leq l,m\leq\dim H_{\beta}$. 

The dual quantum group $\hat{\qg}$ of $\qg$ is defined via its ``algebra
of functions'' 
\[
\ell_{\infty}(\hat{\mathbb{G}})=\oplus_{\alpha\in\irr}B(H_{\alpha})
\]
where $\oplus_{\alpha}B(H_{\alpha})$ refers to the direct sum of
$B(H_{\alpha})$, i.e. the bounded families $(x_{\alpha})_{\alpha}$
with each $x_{\alpha}$ in $B(H_{\alpha})$. We will not completely
recall the quantum group structure on $\hat{\mathbb{G}}$ as we do
not need it in the following. We only remark that the (left) Haar
weight $\hat{h}$ on $\hat{\qg}$ can be explicitly given by (see
e.g. \cite[Section 5]{vandaele1996discrete}) 
\[
\hat{h}:\ell_{\infty}(\hat{\qg})\ni x\mapsto\sum_{\a\in\irr}d_{\a}\tr(Q_{\alpha}p_{\a}x),
\]
where $p_{\a}$ is the projection onto $H_{\a}$ and $\mathrm{Tr}$
denotes the usual trace on $B(H_{\alpha})$ for each $\alpha$. 

Our main result will only concentrate on the case where $\mathbb{G}$
is \emph{of Kac type}, that is, its Haar state is tracial.
\begin{prop}[{\cite[Theorem 1.5]{woronowicz1998note}}]
\label{prop:kac}Let $\mathbb{G}$ be a compact quantum group. The
Haar state $h$ on $C(\mathbb{G})$ is tracial if and only if $Q_{\alpha}=\mathrm{Id}_{H_{\alpha}}$
for all $\alpha\in\irr$ in the formula \eqref{eq:haar state def}
and if and only if the antipode $S$ satisfies $S^{2}(x)=x$ for all
$x\in \mathrm{Pol}(\mathbb{G})$. In particular, if the above conditions are satisfied and $h$ is faithful on $C(\mathbb{G})$, then $S$ extends to a $*$-antihomomorphism on $C(\mathbb{G})$ which is positive and bounded of norm one according to \eqref{antipode}.
\end{prop}

\begin{prop}[\cite{vandaele1997haarfinite}]
If $\mathbb{G}$ is a finite quantum group, then the Haar state is
tracial on $C(\mathbb{G})$.
\end{prop}
For a compact quantum group $\mathbb{G}$, we write $L_{2}(\mathbb{G})$
to be the Hilbert space associated to the GNS-construction with respect
to the Haar state $h$. Similarly, denote by $\ell_{2}(\hat{\mathbb{G}})$
the Hilbert space given by the GNS-construction for $\ell_{\infty}(\hat{\mathbb{G}})$
with respect to the Haar weight $\hat{h}$. If $\mathbb{G}$ is of
Kac type, for $1\leq p\leq\infty$, we denote additionally by $L_{p}(\mathbb{G})$
the $L_{p}$-space associated to the pair $(L_{\infty}(\mathbb{G}),h)$,
as defined in the previous subsection.

Finally we turn to the dual free product of compact quantum groups. The
following construction is given by \cite{wang1995freeprod}.
\begin{prop}
Let $\mathbb{G}_{1}=(A,\Delta_{A})$ and $\mathbb{G}_{2}=(B,\Delta_{B})$
be two compact quantum groups with Haar states $h_{A},h_{B}$ respectively.
There exists a unique comultiplication $\Delta$ on $A*^{c_{0}}B$
such that the pair $(A*^{c_{0}}B,\Delta)$ forms a compact quantum
group, denoted by $\mathbb{G}=\mathbb{G}_{1}\hat{\ast}\mathbb{G}_{2}$
and we have
\[
\Delta|_{A}=(i_{A}\otimes i_{A})\circ\Delta_{A},\quad\Delta|_{B}=(i_{B}\otimes i_{B})\circ\Delta_{B},
\]
where $i_{A}$ and $i_{B}$ are the natural embedding of $A$ and
$B$ into $A*^{c_{0}}B$ respectively. Moreover the Haar state on
$\mathbb{G}$ is the free product state $h_{A}*h_{B}$.
\end{prop}

\subsection{Fourier analysis}

The Fourier transform for locally compact quantum groups has been defined in \cite{kahng2010fourier}, \cite{cooney2010hy} and  \cite{caspers2013fourier}.
In the setting of compact quantum groups, we may give a more explicit
description below. Let a compact quantum group $\mathbb{G}$ be fixed.
For a linear functional $\f$ on $\mathrm{Pol}(\mathbb{G})$, we define
the \emph{Fourier transform} $\hat{\f}=(\hat{\f}(\a))_{\a\in\irr}\in\oplus_{\a}B(H_{\a})$
by 
\begin{equation}
\hat{\f}(\a)=(\f\otimes\iota)((u^{\a})^{*})\in B(H_{\a}),\quad\a\in\irr.
\label{eq:fourier def}
\end{equation}
In particular, any $x\in L_{\infty}(\qg)$ (or $L_{2}(\mathbb{G})$)
induces a continuous functional on $L_{\infty}(\qg)$ defined by $y\mapsto h(yx)$,
and the Fourier transform $\hat{x}=(\hat{x}(\a))_{\a\in\irr}$ of
$x$ is given by 
\[
\hat{x}(\a)=(h(\cdot x)\otimes\iota)((u^{\a})^{*})\in B(H_{\a}),\quad\a\in\irr.
\]
The above definition is slightly different from that of \cite{caspers2013fourier}
or \cite{kahng2010fourier}. Indeed, we replace the unitary $u^{\alpha}$
by $(u^{\a})^{*}$ in the above formulas. This is just to be compatible
with standard definitions in classical analysis on compact groups
such as in \cite[Section 5.3]{folland1995harmonic}, which will not
cause essential difference. On the other hand,
the notation $\hat{\varphi}$ has some slight conflict with the dual
Haar weight $\hat{h}$ on $\hat{\mathbb{G}}$ whereas one can
distinguish them by the elements on which it acts, so we hope that
this will not cause ambiguity for readers. 

Denote by $\mathcal{F}:x\mapsto\hat{x}$ the Fourier transform established
above. It is easy to establish the Fourier inversion formula and the Plancherel theorem
for $L_2(\mathbb{G})$. As we did not find them explicitly for compact quantum groups in the literature, we include the detailed calculation of the Fourier series in the following proposition.
\begin{prop}
\emph{\label{prop:plancherel}(a) }For all $x\in L_{2}(\mathbb{G})$,
we have 
\begin{equation}
x=\sum_{\alpha\in\irr}d_{\alpha}(\iota\otimes\mathrm{Tr})[(1\otimes\hat{x}(\alpha)Q_{\alpha})u^{\alpha}],\label{fourier series}
\end{equation}
where the convergence of the series is in the $L_{2}$-sense. For any $\alpha\in\irr$, if we denote by $\mathcal E _\alpha $ the orthogonal projection of $L_2(\mathbb G )$ onto the subspace spanned by by the matrix coefficients $(u_{ij}^{\alpha})_{i,j=1}^{n_{\alpha}}$, and write $\mathcal{E}_{\alpha}x=\sum_{i,j}x_{ij}^{\alpha}u_{ij}^{\alpha}$ with $x_{ij}^{\alpha}\in\mathbb C$, $X_{\alpha}=[x_{ji}^{\alpha}]_{i,j}$, then
$$\hat x (\alpha)=d_\alpha^{-1}X_\alpha Q_\alpha^{-1}.$$

\emph{(b) }$\mathcal{F}$ is a unitary from $L_{2}(\mathbb{G})$ onto
$\ell_{2}(\hat{\mathbb{G}})$.

\end{prop}
\begin{proof}
(a) Denote by $E_{\alpha}$ the subspace spanned by the matrix units $(u_{ij}^{\alpha})_{i,j=1}^{n_{\alpha}}$
for $\alpha\in\irr$. Then $\mathrm{Pol}(\mathbb{G})$ is spanned
by all the $E_{\alpha}$, $\alpha\in\irr$. It is easy to see from
Hölder's inequality that $\mathrm{Pol}(\mathbb{G})$ is $\|\cdot\|_{2}$-dense
in $L_{\infty}(\mathbb{G})$, and also recall that $L_{2}(\mathbb{G})$
is the $\|\cdot\|_{2}$-completion of $L_{\infty}(\mathbb{G})$, so
$\mathrm{Pol}(\mathbb{G})$ is $\|\cdot\|_{2}$-dense in $L_{2}(\mathbb{G})$
and $L_{2}(\mathbb{G})$ is a Hilbert direct sum of the orthogonal
subspaces $(E_{\alpha})_{\alpha\in\irr}$. So each $x\in L_{2}(\mathbb{G})$
can be written as 
\begin{equation}
x=\sum_{\alpha\in\irr}\mathcal{E}_{\alpha}x=\sum_{\alpha\in\irr}\sum_{i,j}x_{ij}^{\alpha}u_{ij}^{\alpha},\quad(x_{ij}^{\alpha}\in\mathbb{C})\label{eq:in pf decomposition proj}
\end{equation}
where $\mathcal{E_{\alpha}}x\coloneqq\sum_{i,j}x_{ij}^{\alpha}u_{ij}^{\alpha}$
is the orthogonal projection of $x$ onto $E_{\alpha}$.

Now for $\alpha\in\irr$, write $u^{\alpha}=\sum_{l,m}u_{lm}^{\alpha}\otimes e_{lm}^{\alpha}$
and $X_{\alpha}=[x_{ji}^{\alpha}]_{i,j}$ where $e_{lm}^\alpha$ denote the canonical matrix units of $B(H_\alpha)$. Then 
\begin{align}
\hat{x}(\alpha) & =(h(\cdot x)\otimes\iota)((u^{\alpha})^{*})=(h(\cdot\mathcal{E}_{\alpha}x)\otimes\iota)((u^{\alpha})^{*})+(h(\cdot\mathcal{E}_{\alpha}^{\bot}x)\otimes\iota)((u^{\alpha})^{*})\label{eq:fourier coeffecients}\\
 & =\sum_{i,j,l,m}x_{ij}^{\alpha}\left(h(\cdot u_{ij}^{\alpha})\otimes\iota\right)\left((u_{lm}^{\alpha})^{*}\otimes e_{ml}^{\alpha}\right)+0\nonumber \\
 & =\sum_{i,j,l,m}x_{ij}^{\alpha}h\left((u_{lm}^{\alpha})^{*}u_{ij}^{\alpha}\right)e_{ml}^{\alpha}\nonumber \\
 & =d_{\alpha}^{-1}\sum_{i,j,l}x_{ij}^{\alpha}(Q_{\alpha}^{-1})_{il}e_{jl}^{\alpha}=d_{\alpha}^{-1}X_{\alpha}Q_{\alpha}^{-1}.\nonumber 
\end{align}
Hence 
\begin{align*}
  d_{\alpha}(\iota\otimes\mathrm{Tr})[(1\otimes\hat{x}(\alpha)Q_{\alpha})u^{\alpha}] &= \sum_{i,j,l,m}(\iota\otimes\mathrm{Tr})[(1\otimes x_{ji}^{\alpha}e_{ij}^{\alpha})(u_{lm}^{\alpha}\otimes e_{lm}^{\alpha})]\\
& =  \sum_{i,j,m}x_{ji}^{\alpha}(\iota\otimes\mathrm{Tr})(u_{jm}^{\alpha}\otimes e_{im}^{\alpha})
=  \sum_{i,j}x_{ji}^{\alpha}u_{ji}^{\alpha}=\mathcal{E_{\alpha}}x.
\end{align*}
Combining the last equality with \eqref{eq:in pf decomposition proj}
proves the desired (\ref{fourier series}).

(b) Let $$x=\sum_{\alpha\in\irr}\mathcal{E}_{\alpha}x=\sum_{\alpha\in\irr}\sum_{i,j}x_{ij}^{\alpha}u_{ij}^{\alpha}\in L_{2}(\mathbb{G}).$$
For each $\alpha\in\irr$, 
\begin{align*}
\|\mathcal{E_{\alpha}}x\|_{2}^{2} & =h((\mathcal{E_{\alpha}}x)^{*}(\mathcal{E_{\alpha}}x))=\sum_{i,j,l,m}\overline{x_{ij}^{\alpha}}x_{lm}^{\alpha}h\left((u_{ij}^{\alpha})^{*}u_{lm}^{\alpha}\right)\\
 & =d_{\alpha}^{-1}\sum_{i,j,l}\overline{x_{ij}^{\alpha}}x_{lj}^{\alpha}(Q_{\alpha}^{-1})_{li}
\end{align*}
and also by (\ref{eq:fourier coeffecients}) 
\begin{align*}
d_{\alpha}^{2}\mathrm{Tr}(Q_{\alpha}\hat{x}(\alpha)^{*}\hat{x}(\alpha)) & =\mathrm{Tr}(X_{\alpha}^{*}X_{\alpha}Q_{\alpha}^{-1})=\sum_{i,j,l}\overline{x_{ij}^{\alpha}}x_{lj}^{\alpha}(Q_{\alpha}^{-1})_{li}.
\end{align*}
Hence by Parseval's identity, 
\begin{align*}
\|x\|_{2}^{2} & =\sum_{\alpha}\|\mathcal{E_{\alpha}}x\|_{2}^{2}=\sum_{\alpha}d_{\alpha}^{-1}\sum_{i,j,l}\overline{x_{ij}^{\alpha}}x_{lj}^{\alpha}(Q_{\alpha}^{-1})_{li}\\
 & =\sum_{\alpha}d_{\alpha}^{-1}d_{\alpha}^{2}\mathrm{Tr}(Q_{\alpha}\hat{x}(\alpha)^{*}\hat{x}(\alpha))=\|\hat{x}\|_{2}^{2}.
\end{align*}
Thus $\mathcal{F}$ maps isometrically $L_{2}(\mathbb{G})$ into $\ell_{2}(\hat{\mathbb{G}})$.
From (\ref{fourier series}) and the isometric relation we see that
the range of $\mathcal{F}$ contains the subset of all finitely supported
families $(a_{\alpha})\in\oplus_{\alpha}B(H_{\alpha})$, which is
dense in $\ell_{2}(\hat{\mathbb{G}})$. Therefore $\mathcal{F}$ is
surjective and hence unitary. 
\end{proof}

\begin{example}
(1) Let $G$ be a compact group and define
$$\Delta(f)(s,t)=f(st),\quad f\in C(G),\ s,t\in G.$$ 
Then $\mathbb{G}=(C(G),\Delta)$ is a compact quantum group. The elements in $ \mathrm{Irr}(\mathbb{G})\coloneqq \mathrm{Irr}(G) $ coincide with the usual strongly continuous irreducible unitary representations of $G$. Any continuous functional $\varphi$ on $C(G)$ corresponds to a complex Radon measure $\mu$ on $G$ by the Riesz representation theorem. By definition \eqref{eq:fourier def}, the Fourier series of $\mu$ is given by 
$$\hat{\mu}(\pi)=(\varphi\otimes\iota)(\pi(\cdot)^*)=\int_G \pi(g)^*\,d \mu(g),\quad \pi\in \mathrm{Irr}(G).$$
In particular for $f\in L_2(G)$, we have 
$$\hat{f}(\pi)=\int_G\pi(g)^*f(g)dg,\quad \pi\in \mathrm{Irr}(G)$$
and we have the Fourier expansion and the Plancherel formula
$$f=\sum_{\pi\in \mathrm{Irr}(G)}d_\pi\mathrm{Tr}(\hat{f}(\pi)\pi),\quad \|f\|_2^2=\sum_{\pi\in \mathrm{Irr}(G)}d_\pi \|\hat{f}(\pi)\|_{\mathrm{HS}}^2$$
where $d_\pi$ is the dimension of the Hilbert space on which the representation $\pi$ acts and $\|\|_{\mathrm{HS}}$ denotes the usual Hilbert-Schmidt norm. We refer to \cite[Section 5.3]{folland1995harmonic} and  \cite[pp.77-87]{hewittross1970abstract} for more information.

(2) Let $\Gamma$ be a discrete group with its neutral element $e$ and $C^{*}_r(\Gamma)$ be the associated reduced group
C{*}-algebra generated by $\lambda(\Gamma)\subset B(\ell_{2}(\Gamma))$,
where $\lambda$ denotes the left regular representation. The \textquotedblleft dual" $\mathbb{G}=\hat{\Gamma}$ of $\Gamma$ is a compact quantum group such that $C(\mathbb{G})$ is the group C{*}-algebra $C^{*}_r(\Gamma)$ equipped
with the comultiplication $\Delta:C^{*}_r(\Gamma)\to C^{*}_r(\Gamma)\otimes C^{*}_r(\Gamma)$
defined by
\[
\Delta(\lambda(\gamma))=\lambda(\gamma)\otimes\lambda(\gamma),\quad\gamma\in\Gamma.
\]
The Haar state of $\mathbb{G}$ is the unique trace $\tau$ on $C^{*}_r(\Gamma)$ such that $\tau(1)=1$ and $\tau(\lambda(\gamma))=0$ for $\gamma\in \Gamma\setminus\{e\}$. The elements of $\lambda(\Gamma)$ give all irreducible unitary representations of $\mathbb{G}$, which are all of dimension $1$. It is easy to check from definition that for any $f\in C^*_r(\Gamma)$,
\[ \hat f (\gamma)= \tau(f\lambda(\gamma)^*),\quad \gamma\in \Gamma.\]
And any $f\in L_2(\mathbb{G})$ has an expansion such that 
$$f=\sum_{\gamma\in \Gamma}\hat{f}(\gamma)\lambda(\gamma),\quad \|f\|_2^2=\sum_{\gamma\in \Gamma}|\hat{f}(\gamma)|^2.$$
\end{example}

Let us end the section with a brief description of multipliers and convolutions in terms of Fourier series. Return back to a general compact quantum group $\mathbb{G}$. For $a=(a_{\a})_{\a}\in\oplus_{\a}B(H_{\a})$, we define the left
\emph{multiplier} $m_{a}:\mathrm{Pol}(\qg)\to\mathrm{Pol}(\qg)$ associated
to $a$ by 
\begin{equation}
m_{a}x=\sum_{\a\in\irr}d_{\a}(\iota\otimes\tr)[(1\otimes\hat{x}(\a)a_{\a}Q_{\alpha})u^{\a}],\quad x\in\mathrm{Pol}(\qg),\label{left multiplier formula}
\end{equation}
which means that 
\begin{equation}
(m_{a}x)\,\hat{}\,(\alpha)=\hat{x}(\alpha)a_{\alpha},\quad\alpha\in\irr.\label{eq:multiplier, Fourier coefficient}
\end{equation}
In the same way we may define the right multiplier $m_{a}':\mathrm{Pol}(\qg)\to\mathrm{Pol}(\qg)$
such that 
\[
m_{a}'x=\sum_{\a\in\irr}d_{\a}(\iota\otimes\tr)[(1\otimes a_{\a}\hat{x}(\a)Q_{\alpha})u^{\a}],\quad x\in\mathrm{Pol}(\qg).
\]
Observe that the multiplier $m_a$ (or $m_a'$ resp.) is unital, i.e. $m_a(1)=1$ ($m_a'(1)=1$ resp.) if and only if $a_1=1$.
\begin{rem}
In case $\mathbb{G}$ is of Kac type, that is, $Q_{\alpha}=\mathrm{Id}_{H_{\alpha}}$
for all $\alpha\in\irr$ by Proposition \ref{prop:kac}, the multipliers
$m_{a}$ and $m_{a}'$ can be equivalently defined by 
\begin{equation}
(m_{a}\otimes\iota)(u^{\a})=(1\otimes a_{\a})(u^{\a}),\quad(m_{a}'\otimes\iota)(u^{\a})=(u^{\a})(1\otimes a_{\a}),\quad\a\in\irr,
\label{multiplier def by junge et al}
\end{equation}
which corresponds to the standard definition of left and right multipliers
on locally compact quantum groups in \cite{jungeruan2009repquantumgroup}
and \cite{daws2012cpmultiplier}. If $\mathbb{G}$ is not of Kac type,
the above formula \eqref{multiplier def by junge et al} gives a similar equality corresponding to \eqref{eq:multiplier, Fourier coefficient},
that is, $(m_{a}x)\,\hat{}\,(\alpha)=\hat{x}(\alpha)Q_{\alpha}a_{\alpha}Q_{\alpha}^{-1}$,
$\alpha\in\irr$.
\end{rem}

We will use the standard definition of convolution products given by Woronowicz. Let $x\in C(\mathbb{G})$ and $\f,\f'$ be linear functionals on $C(\mathbb{G})$. We define  
\begin{align*}
\f\star\f'&=(\f\otimes\f')\circ\Delta,\\
x\star \f &=(\varphi\otimes\iota)\Delta(x),\\
\f\star x
&=(\iota\otimes\varphi)\Delta(x).
\end{align*}
Observing the embedding $x\mapsto h(\cdot\, x)$ from $C(\mathbb{G})$ into $C(\mathbb{G})^{*}$, the convolution products defined above are related as follows according to \eqref{antipode property} (see also \cite[Proposition 2.2]{vandaele2007fourier}): on $\mathrm{Pol}(\mathbb{G})$ we have
\begin{equation}\label{two convolutions}
h(\cdot\, x)\star \f=h(\cdot\, [(\f\circ S)\star x]),\quad \f\star h(\cdot\, x)=h(\cdot\, [x\star (\f\circ S^{-1})]).
\end{equation}
We note that for $\a\in\irr$ and $u^{\a}=[u_{ij}^{\a}]_{1\leq i,j\leq n_{\a}}$,
\[
\left[\f((u_{ji}^{\a})^{*})\right]_{i,j}=(\f\otimes\iota)((u^{\a})^{*})=\hat{\f}(\a).
\]
Then by \eqref{eq:comultiplication}, a straightforward calculation
shows that \begin{equation}\label{fourier of conv of functional}
(\f\star\f')\,\hat{}\,(\a)=\hat{\f'}(\a)\hat{\f}(\a).
\end{equation}
Hence together with \eqref{two convolutions},
\begin{equation}\label{eq:fourier series for convolution}
(x\star\f)\,\hat{}\,(\a)=\hat{x}(\a)(\f\circ S)\,\hat{}\,(\a),\quad
(\f\star x)\,\hat{}\,(\a)=(\f\circ S^{-1})\,\hat{}\,(\a)\hat{x}(\a).
\end{equation}

\section{\label{sec:convolution and hopf image}Non-degenerate states and applications to Hopf images}

In this short section we give the key lemma on non-degenerate states, which will be of use for our main results. We need
the following observation adapted from \cite[Lemma 2.1]{woronowicz1998note}. The result is mentioned in \cite{soltan2005bohr}.
\begin{lem}
Let $\mathbb{G}$ be a compact quantum group and $A=C(\mathbb{G})$.
Suppose that $(\rho_{i})_{i\in I}$ is a family of states on $A$
separating the points of $A_{+}$, i.e., $\forall x\in A_{+}\backslash\{0\},$
$\exists i\in I$, $\rho_{i}(x)>0$. If $\rho$ is a state on $A$
such that 
\[
\forall i\in I,\quad\rho\star\rho_{i}=\rho_{i}\star\rho=\rho,
\]
then $\rho$ is the Haar state $h$ of $\mathbb{G}$.\end{lem}
\begin{proof}
Set 
\[
\mathcal{I}=\{q\in A\otimes A:\forall i\in I,\,(\rho_{i}\otimes\rho)(q^{*}q)=0\}.
\]
Then $\mathcal{I}$ is a closed left ideal of $A\otimes A$. Define
\[
\Psi_{L}(x)=(\rho\otimes\iota)\Delta(x)-\rho(x)1,\quad x\in A.
\]
Since $\Psi_{L}$ is a difference of two unital completely positive
maps, we see that $\Psi_{L}$ is a completely bounded map with norm
at most $2$. We will prove that 
\[
(\Psi_{L}\otimes\iota)\Delta(A)\subset\mathcal{I}.
\]
In fact, given $x\in A$, by the coassociativity of $\Delta$ we have
\begin{align*}
q & \coloneqq(\Psi_{L}\otimes\iota)\Delta(x)=(\rho\otimes\iota\otimes\iota)(\iota\otimes\Delta)\Delta(x)-1\otimes[(\rho\otimes\iota)\Delta(x)]\\
 & =\Delta((\rho\otimes\iota)\Delta(x))-1\otimes[(\rho\otimes\iota)\Delta(x)].
\end{align*}
Thus
\begin{align*}
q^{*}q & =\Delta\left([(\rho\otimes\iota)\Delta(x)]{}^{*}(\rho\otimes\iota)\Delta(x)\right)-\Delta((\rho\otimes\iota)\Delta(x))^{*}[1\otimes(\rho\otimes\iota)\Delta(x)]\\
 & \qquad-[1\otimes\left((\rho\otimes\iota)\Delta(x)\right)^{*}]\Delta((\rho\otimes\iota)\Delta(x))+1\otimes\left([(\rho\otimes\iota)\Delta(x)]^{*}[(\rho\otimes\iota)\Delta(x)]\right)
\end{align*}
and hence for any $i\in I$ we may write
\[
(\rho_{i}\otimes\rho)(q^{*}q)=q_{1}-q_{2}-q_{3}+q_{4}
\]
where by the convolution invariance assumption and the coassociativity
of $\Delta$ we have
\begin{align*}
q_{1} & =(\rho_{i}\otimes\rho)\Delta\left([(\rho\otimes\iota)\Delta(x)]{}^{*}(\rho\otimes\iota)\Delta(x)\right)=\rho\left([(\rho\otimes\iota)\Delta(x)]{}^{*}(\rho\otimes\iota)\Delta(x)\right),\\
q_{2} & =q_{3}^{*},\\
q_{3} & =(\rho_{i}\otimes\rho)\left([1\otimes\left((\rho\otimes\iota)\Delta(x)\right)^{*}]\Delta((\rho\otimes\iota)\Delta(x))\right)\\
 & =\rho\left([\left((\rho\otimes\iota)\Delta(x)\right)^{*}](\rho_{i}\otimes\iota)((\rho\otimes\iota\otimes\iota)(\iota\otimes\Delta)\Delta(x))\right)\\
 & =\rho\left([\left((\rho\otimes\iota)\Delta(x)\right)^{*}](\rho\otimes\rho_{i}\otimes\iota)((\Delta\otimes\iota)\Delta(x))\right)=\rho\left([(\rho\otimes\iota)\Delta(x)]{}^{*}(\rho\otimes\iota)\Delta(x)\right),\\
q_{4} & =(\rho_{i}\otimes\rho)\left(1\otimes[\left((\rho\otimes\iota)\Delta(x)\right)^{*}(\rho\otimes\iota)\Delta(x)]\right)=\rho\left([(\rho\otimes\iota)\Delta(x)]{}^{*}(\rho\otimes\iota)\Delta(x)\right).
\end{align*}
Note that $q_{1}=q_{2}=q_{3}=q_{4}$. So $(\rho_{i}\otimes\rho)(q^{*}q)=0$
and $(\Psi_{L}\otimes\iota)\Delta(A)\subset\mathcal{I}$ is proved.

Now by the density of $(1\otimes A)\Delta(A)$ in $A\otimes A$ and
the complete boundedness of $\Psi_{L}$, it follows that $\Psi_{L}(A)\otimes1\subset\overline{(1\otimes A)(\Psi_{L}\otimes\iota)\Delta(A)}$
is also contained in the closed left ideal $\mathcal{I}$, which means
that for any $i\in I$ and $x\in A$,
\[
\rho_{i}(\Psi_{L}(x)^{*}\Psi_{L}(x))=\rho_{i}\otimes\rho(\Psi_{L}(x)^{*}\Psi_{L}(x)\otimes1)=0.
\]
Recall that $(\rho_{i})_{i\in I}$ separates the points of $A_{+}$,
so we have $\Psi_{L}(x)=0$ and $(\rho\otimes\iota)\Delta(x)=\rho(x)1$
for all $x\in A$.

A similar argument applies as well to the map $\Psi_{R}(x)=(\iota\otimes\rho)\Delta(x)-\rho(x)1$,
$x\in A$. So $\rho=h$ is the Haar state.\end{proof}
\begin{rem}
We remark that in case $\mathbb{G}$ is a finite quantum group,
we can provide a simpler proof of the above lemma. Indeed, since $C(\mathbb{G})$
is finite-dimensional, its dual space $C(\mathbb{G})^{*}$ is also
finite-dimensional, and we can take a maximal linear independent family
$\{\rho_{i_{1}},\ldots,\rho_{i_{s}}\}\subset(\rho_{i})_{i\in I}$
which form a basis of the subspace spanned by $(\rho_{i})_{i\in I}$
in $C(\mathbb{G})^{*}$. Given a nonzero $x\in A_{+}$, there is an
$i\in I$ such that $\rho_{i}(x)>0$. Write $\rho_{i}=\sum_{k=1}^{s}a_{k}\rho_{i_{k}}\ (a_{k}\in\mathbb{C})$,
then we see clearly that there exists at least one $k\in\{1,\ldots,s\}$
such that $\rho_{i_{k}}(x)\neq0$ in order that $\rho_{i}(x)>0$.
Then $\rho'=\frac{1}{s}\sum_{k=1}^{s}\rho_{i_{k}}$ is faithful on
$C(\qg)$ and $\rho\star\rho'=\rho'\star\rho=\rho$, thus $\rho$
is the Haar state by \cite[Lemma 2.1]{woronowicz1998note}.
\end{rem}
We immediately obtain the following fact.
\begin{lem}
\label{lem:non degenerate}Let $\mathbb{G}$ be a compact quantum
group and $\varphi$ be a state on $C(\mathbb{G})$. If $\varphi$ is \emph{non-degenerate} on $C(\mathbb{G})$ in
the sense that for all nonzero $x\in C(\mathbb{G})_{+}$ there exists
$k\geq0$ such that $\varphi^{\star k}(x)>0$, then 
\[
w^{*}\text{-}\lim_{n\to\infty}\frac{1}{n}\sum_{k=1}^{n}\varphi^{\star k}=h.
\]
If additionally $h$ is faithful on $C(\mathbb{G})$, then the converse also holds.\end{lem}
\begin{proof}
If the above limit holds and $h$ is faithful on $C(\mathbb{G})$, then clearly $\varphi$ is non-degenerate
since if there existed a nonzero $x\geq0$ such that $\varphi^{\star n}(x)=0$
for all $n$, then we would have $\lim_{n}\frac{1}{n}\sum_{k=1}^{n}\varphi^{\star k}(x)=0$,
which contradicts the faithfulness of $h$. On the other hand, if
$\varphi$ is non-degenerate, the family of states $\{\frac{1}{n}\sum_{k=1}^{n}\varphi^{\star k}:n\geq1\}$
separates the points of $C(\mathbb{G})_{+}$, so any accumulation
point of this family becomes the unique Haar state by our previous
lemma.
\end{proof}

We would like to digress momentarily to see an application of above lemmas to some problems concerning Hopf images. Let $A$ be a unital C*-algebra and $\pi:C(\mathbb{G})\to A$ be a unital $*$-homomorphism. The \emph{Hopf image} of $\pi$, firstly introduced by Banica and Bichon in \cite{banicabichon2010hopfimage}, is the largest algebra $C(\mathbb{G}_\pi)$ for some compact quantum subgroup $\mathbb{G}_\pi\subset \mathbb{G}$ such that $\pi$ factorizes through $C(\mathbb{G}_\pi)$. In this paper we will use the following equivalent characterization recently given in \cite{skalskisoltan2014quantumfamily}: for each $k$ consider
$$
\pi_k =(\underbrace{\pi \otimes \cdots \otimes \pi}_k) \circ \Delta^{(k-1)}:C(\mathbb{G})\to A^{\otimes k}
$$
and let $\mathcal{I}=\cap_{k=1}^\infty \ker \pi_k$. Then, there is a compact quantum group $\mathbb{G}_\pi=(B,\Delta_\pi)$ with a $*$-homomorphism $\pi_q:B\to A$ such that 
\begin{equation}\label{eq: hopf image}
B=C(\mathbb{G})/\mathcal{I},\quad \Delta_\pi \circ q=(q\otimes q)\circ \Delta,\quad \pi=\pi_q \circ q
\end{equation}
where $q:C(\mathbb{G})\to C(\mathbb{G})/\mathcal{I}$ denotes the quotient map. The algebra $B=C(\mathbb{G}_\pi )$ is exactly the Hopf image of $\pi$. Now let $h_\mathbb{G}, h_{\mathbb{G}_\pi}$ be the Haar states on $\mathbb{G},\mathbb{G}_\pi$ respectively. A related question raised in \cite{banicafranzskalski2012inner} is the computation of the associated idempotent state $h_{\mathbb{G}_\pi}\circ q$ on $\mathbb{G}$. Simply based on Lemma \ref{lem:non degenerate}, the following property generalizes the main result of \cite{banicafranzskalski2012inner}.

\begin{thm}
Let $\mathbb{G}$ be a compact quantum group and $A$ be a unital C*-algebra with a unital $*$-homomorphism $\pi:C(\mathbb{G})\to A$. Let $\mathbb{G}_\pi$ be the compact quantum group constructed above. Then given any faithful state $\varphi$ on $A$, 
$$ h_{\mathbb{G}_\pi}\circ q= w^*\text{-}\lim_{n\to\infty}\frac{1}{n}\sum_{k=1}^{n}(\varphi\circ \pi)^{\star k}.$$
\end{thm}
\begin{proof}
Let $\varphi$ be a faithful state on $A$ and let $\mathcal{I}=\cap_{k=1}^\infty \ker \pi_k$ with $\pi_k,k\geq 1$ constructed as above. Let us show that $\varphi\circ \pi_q$ is non-degenerate on $B=C(\mathbb{G}_\pi)$. Consider any $x=q(y)$ with $y\in C(\mathbb{G})_+$ satisfying
$$\forall k\geq 1,\quad (\varphi\circ \pi_q)^{\star k}(x)=0.$$ 
Since, by \eqref{eq: hopf image},
\begin{align*}
(\varphi\circ \pi_q)^{\star k}&=[(\varphi\otimes\cdots\otimes\varphi)\circ (\pi_q\otimes\circ\otimes \pi_q)]\Delta^{(k-1)}(x)\\
&=[(\varphi\otimes\cdots\otimes\varphi)\circ ((\pi_q\circ q)\otimes\cdots\otimes (\pi_q\circ q) )]\Delta^{(k-1)}(y)=(\varphi\otimes\cdots\otimes\varphi)(\pi_k (y))
\end{align*}
and since $\varphi$ is faithful, we get $y\in \cap_{k=1}^\infty \ker \pi_k$, which means that $x=q(y)=0$. As a result $\varphi\circ \pi_q$ is non-degenerate. Therefore we have $h_{\mathbb{G}_\pi}= w^*\text{-}\lim_{n\to\infty}\frac{1}{n}\sum_{k=1}^{n}(\varphi\circ \pi_q)^{\star k}$, and hence using \eqref{eq: hopf image} again,
$$h_{\mathbb{G}_\pi}\circ q= w^*\text{-}\lim_{n\to\infty}\frac{1}{n}\sum_{k=1}^{n}(\varphi\circ \pi_q)^{\star k}\circ q=w^*\text{-}\lim_{n\to\infty}\frac{1}{n}\sum_{k=1}^{n}(\varphi\circ \pi)^{\star k},$$
as desired.
\end{proof}

\section{\label{sec:Some--improving-convolutions}Main results}

In this section we aim to give several characterizations of
$L_{p}$-improving convolutions given by states on finite quantum groups,
and also give the constructions for the free product of finite quantum
groups. We will start with some discussions on multipliers on
compact quantum groups. 
In this section we keep the notation of multipliers $m_{a}$, $m_{a}'$
and convolutions $\varphi_{1}\star\varphi_{2}$ given in Section \ref{sec:Fourier-analysis}.
\begin{lem}
Let $\mathbb{G}$ be a compact quantum group of Kac type. Suppose $a\in\ell_{\infty}(\hat{\qg})$
such that $m_{a}$ (resp. $m_{a}'$) extends to a unital left (resp.,
right) multiplier on $L_2(\qg)$ and $b=aa^{*}$. Then $\lim_{n}\frac{1}{n}\sum_{k=1}^{n}m_{b}^{k}x=h(x)1$
for all $x\in L_2(\qg)$ if and only if $\lim_{n}\frac{1}{n}\sum_{k=1}^{n}m_{b}'^{k}x=h(x)1$
for all $x\in L_2(\qg)$ if and only if $\nm{a_{\a}}<1$ for all $\a\in\irr\setminus\set1$.\end{lem}
\begin{proof}
Without loss of generality we only discuss the left multiplier $m_{a}$.
Assume that $\nm{a_{\a}}<1$ for all $\a\in\irr\setminus\set1$. By
the Plancherel theorem \ref{prop:plancherel} and the formula \eqref{left multiplier formula}
we note that $m_{b}$ extends to a bounded map of norm one on $L_{2}(\mathbb{G})$.
We first consider the case $x\in\mathrm{Pol}(\mathbb{G})$, so that
$\hat{x}$ is finitely supported. Let $\a\in\irr\setminus\set1$ and
$\nm{a_{\a}}<1$. Then 
\[
\nm{(\frac{1}{n}\sum_{k=1}^{n}m_{b}^{k}x)\,\hat{}\,(\a)}_{2}=\nm{\frac{1}{n}\sum_{k=1}^{n}\hat{x}(\a)(a_{\a}a_{\a}^{*})^{k}}_{2}\leq\frac{1}{n}\sum_{k=1}^{n}\nm{a_{\a}}^{2k}\nm{\hat{x}(\a)}_{2}\to0
\]
whenever $n\to\infty$. And for $\a=1$, 
\[
(\frac{1}{n}\sum_{k=1}^{n}m_{b}^{k}x)\,\hat{}\,(1)=\hat{x}(1)=h(x).
\]
Thus by the Plancherel theorem 
\[
\nm{\frac{1}{n}\sum_{k=1}^{n}m_{b}^{k}x-h(x)1}_{2}^{2}=\sum_{\a\neq1}d_{\a}\nm{(\frac{1}{n}\sum_{k=1}^{n}m_{b}^{k}x)\,\hat{}\,(\a)}_{2}^{2}\to0
\]
when $n\to\infty$. Since $\frac{1}{n}\sum_{k=1}^{n}m_{b}^{k}$ is a contraction on $L_2(\mathbb{G})$ and $\mathrm{Pol}(\mathbb{G})$ is dense in $L_2(\mathbb{G})$, we get the convergence $\lim_{n}\frac{1}{n}\sum_{k=1}^{n}m_{b}^{k}x=h(x)1$
for all $x\in L_2(\qg)$.

Conversely, if $\exists\a_{0}\in\irr\backslash\{1\}$, $\nm{a_{\a_{0}}}=1$,
then viewing $b_{\a_{0}}$ as a matrix in $\mathbb{M}_{n_{\a_{0}}}$,
we observe that $1\in\sigma(b_{\a_{0}})$ and there exists a nonzero
$x_{\a_{0}}\in\mathbb{M}_{n_{\a_{0}}}$ such that $x_{\a_{0}}b_{\a_{0}}=x_{\a_{0}}$.
Take $x\in L_2(\qg)$ such that $\hat{x}(1)=1$, $\hat{x}(\a_{0})=x_{\a_{0}}$,
$\hat{x}(\a)=0$ for $\a\in\irr\backslash\{1,\a_{0}\}$. Then $m_{b}x=x$
and hence $\frac{1}{n}\sum_{k=1}^{n}m_{b}^{k}x=x$ does not converge
to $h(x)1$. 
\end{proof}
\begin{rem}
In case the compact quantum group $\mathbb{G}$ is not of Kac type, the above argument still remains true for right multipliers. 
\end{rem}
The following first main result is now in reach. We will consider
the case where $\mathbb{G}$ is a finite quantum group. 
\begin{thm}
\label{thm:multiplier_quantum}Let $\mathbb{G}$ be a finite quantum
group. Suppose $a\in\ell_{\infty}(\hat{\qg})$ is such that $m_{a}$
(resp., $m_{a}'$) is a unital left (resp., right) multiplier on $C(\qg)$.
Then the following assertions are equivalent: 
\begin{enumerate}[label=\emph{(\arabic{enumi})}]
\item there exists $1\leq p<2$ such that, 
\[
\forall\ x\in C(\qg),\ \nm{m_{a}x}_{2}\leq\nm{x}_{p}\ ;
\]

\item there exists $1\leq p<2$ such that, 
\[
\forall\ x\in C(\qg),\ \nm{m_{a}'x}_{2}\leq\nm{x}_{p}\ ;
\]

\item $\nm{a_{\a}}<1$ for all $\a\in\irr\setminus\set1$ ; 
\item  $\lim_{n}\frac{1}{n}\sum_{k=1}^{n}m_{b}^{k}x=h(x)1$ for all $x\in C(\qg)$ when $b=aa^{*}$;
\item  $\lim_{n}\frac{1}{n}\sum_{k=1}^{n}m_{b}'^{k}x=h(x)1$ for all $x\in C(\qg)$ when $b=aa^{*}$.
\end{enumerate}
\end{thm}
\begin{proof}
Without loss of generality we only discuss the left multiplier $m_{a}$
and prove the equivalence (1)$\Leftrightarrow$(3)$\Leftrightarrow$(4). 

It is easy to see from Plancherel's theorem that (3) is just \eqref{eq:spetral gap}
for $T=m_{a}$. In fact note that for $x\in C(\mathbb{G})$, $h(x)=0$
if and only if $\hat{x}(1)=0$, so (3) implies \eqref{eq:spetral gap}
via Plancherel's theorem. On the other hand, suppose by contradiction
that there exists $\a\in\irr\setminus\set1$ such that $\nm{a_{\a}}=1$.
By Proposition \ref{prop:plancherel} we may take a nonzero $x\in C(\qg)$
such that $\hat{x}(1)=0$, $\hat{x}(\a)=0$ when $\nm{a_{\a}}<1$,
and $\nm{\hat{x}(\a)a_{\a}}_{2}=\nm{\hat{x}(\a)}_{2}$ when $\nm{a_{\a}}=1$.
Then $\|m_{a}x\|_{2}=\|x\|_{2}$ with $h(x)=0$. As a consequence
the equivalence (1)$\Leftrightarrow$(3) follows from Theorem \ref{Lp improvement_spectral gap_strong version}.

The equivalence between (3) and (4) was proved in the previous lemma.
Therefore the theorem is established. 
\end{proof}

Now we turn to the corresponding convolution problems. Let $\f\in C(\mathbb{G})^{*}$ for a compact quantum group $\mathbb{G}$.
Recall the formula \eqref{fourier of conv of functional}, and
then we have 
\begin{equation}
\left[\f^{\star n}((u_{ji}^{\a})^{*})\right]=(\f^{\star n})\,\hat{}\,(\a)=\hat{\f}(\a)^{n}.\label{eq:Fourier coefficient of functional}
\end{equation}
Note that the convergence $\frac{1}{n}\sum_{k=1}^{n}m_{\hat{\f}^{*}}^{k}(x)\to h(x)1$
for all $x\in\mathrm{Pol}(\mathbb{G})$, by (\ref{eq:multiplier, Fourier coefficient})
can be reformulated in terms of Fourier coefficients as
\[
\left(\frac{1}{n}\sum_{k=1}^{n}m_{\hat{\f}^{*}}^{k}(x)\right)\,\hat{}\,(1)=h(x)\hat{\varphi}(1)^{*}=h(x)1,
\]
\[
\lim_{n}\left(\frac{1}{n}\sum_{k=1}^{n}m_{\hat{\f}^{*}}^{k}(x)\right)\,\hat{}\,(\alpha)=\lim_{n}\frac{1}{n}\sum_{k=1}^{n}(\hat{\varphi}(\alpha)^*)^{k}\hat{x}(\alpha)=0,\quad\alpha\in\irr\backslash\{1\}.
\]
This is to say,
\[
\hat{\varphi}(1)=1,\quad\lim_{n}\frac{1}{n}\sum_{k=1}^{n}\hat{\varphi}(\alpha)^{k}=0,\quad\alpha\in\irr\backslash\{1\},
\]
which, according to \eqref{eq:Fourier coefficient of functional},
is equivalent to $\frac{1}{n}\sum_{k=1}^{n}\f^{\star k}(u_{ij}^{\alpha})\to h(u_{ij}^{\alpha})$
for all $\alpha\in\irr$ and $1\leq i,j\leq n_{\alpha}$, or in other
words, 
\begin{equation*}
\frac{1}{n}\sum_{k=1}^{n}\f^{\star k}(x)\to h(x),\quad n\to\infty,\quad x\in\mathrm{Pol}(\mathbb{G}).
\end{equation*}

Any state $\f$ on $C(\qg)$ induces two convolution operators on
$C(\qg)$ 
\[
T_{\f}:C(\mathbb{G})\ni x\mapsto x\star\varphi=(\varphi\otimes\iota)\Delta(x),\quad T_{\varphi}':C(\mathbb{G})\ni x\mapsto\f\star x=(\iota\otimes\varphi)\Delta(x).
\]
If additionally $\mathbb{G}$ is of Kac type and the Haar state is faithful on $C(\mathbb{G})$, then by Proposition \ref{prop:kac} the antipode $S$ extends to a positive linear operator on $C(\mathbb{G})$ and $S=S^{-1}$, and hence $\varphi\circ S=\varphi\circ S^{-1}$ is also a state. In this case we have
\begin{equation*}
\left[\,\overline{\f(u_{ji}^{\a})}\,\right]_{i,j}=\hat{\f}(\a),\quad(\f\circ S)\,\hat{}\,(\a)=\left[\,\f(u_{ij}^{\a})\,\right]_{i,j}=\hat{\f}(\a)^{*}
\end{equation*}
and by \eqref{eq:fourier series for convolution} we have $(x\star\varphi)\,\hat{}\,(\a)=\hat{x}(\a)\hat{\f}(\a)^{*}$
and $(\f\star x)\,\hat{}\,(\a)=\hat{\f}(\a)^{*}\hat{x}(\a)$ for $\a\in\irr$, $x\in C(\mathbb{G})$.
So $T_{\f}=m_{\hat{\f}^{*}}$ and $T_{\varphi}'=m_{\hat{\varphi}^{*}}'$
are unital completely positive left and right multipliers, respectively. 

Now with these remarks and Lemma \ref{lem:non degenerate} in hand,
we may reformulate Theorem \ref{thm:multiplier_quantum} in terms of
convolution operators using the above arguments. 
\begin{thm}
\label{cor:convolution_quantul}Let $\mathbb{G}$ be a finite quantum
group and $\f$ be a state on $C(\qg)$. Denote $\psi=(\f\circ S)\star\f$.
The following assertions are equivalent: 
\begin{enumerate}[label=\emph{(\arabic{enumi})}]
\item there exists $1\leq p<2$ such that, 
\[
\forall\ x\in C(\qg),\ \nm{\f\star x}_{2}\leq\nm{x}_{p}\ ;
\]

\item there exists $1\leq p<2$ such that, 
\[
\forall\ x\in C(\qg),\ \nm{x\star\varphi}_{2}\leq\nm{x}_{p}\ ;
\]

\item $\nm{\hat{\f}(\a)}<1$ for all $\a\in\irr\setminus\set1$ ; 
\item $\lim_{n}\frac{1}{n}\sum_{k=1}^{n}\psi^{\star k}=h$; 
\item For any nonzero $x\in C(\qg)_{+}$, there exists $n\geq1$ such that
$\psi^{\star n}(x)>0$. 
\end{enumerate}
\end{thm}
Note that if $C(\qg)$ is commutative, i.e. $C(\qg)=C(G)$ where $G$
is a finite group, then $\f\in C(G)^{*}$ corresponds to
a Radon measure $\mu$ via the Riesz representation theorem. The above
condition (5) in the theorem just asserts that $G$ is the union of
$D_{n}\coloneqq\text{supp}\,(\nu^{\star n})$, $n\geq1$, where $\nu$
denotes the Radon measure corresponding to $\psi$. It is easy to
see that $D_{1}=\set{i^{-1}j:i,j\in\text{supp}\,(\mu)}$ and $D_{n}=D_{1}^{n}$.
So the above corollary covers Ritter's result \cite{ritter1984convolution}. 
\begin{cor}
\label{cor:classical ritter's thm}Let $G$ be a finite group and
$\mu$ be a probability measure on $G$. Then there is a $1\leq p<2$
such that 
\[
\nm{x\star\mu}_{2}\leq\nm x_{p},\quad x\in L_{p}(G)
\]
if and only if $G$ is equal to the subgroup generated by $\set{i^{-1}j:i,j\in\mathrm{supp}\,(\mu)}$. 
\end{cor}
On the other hand, let $\Gamma$ be a finite group with
neutral element $e$ and $C^{*}(\Gamma)$ be the associated group
C{*}-algebra generated by $\lambda(\Gamma)\subset B(\ell_{2}(\Gamma))$,
where $\lambda$ denotes the left regular representation. Recall that if $\mathbb{G}=\hat{\Gamma}$,
then $C(\mathbb{G})$ is the group C{*}-algebra $C^{*}(\Gamma)$ equipped
with the comultiplication $\Delta:C^{*}(\Gamma)\to C^{*}(\Gamma)\otimes C^{*}(\Gamma)$
defined by
\[
\Delta(\lambda(\gamma))=\lambda(\gamma)\otimes\lambda(\gamma),\quad\gamma\in\Gamma.
\]
Note that any state $\Phi$ on $C(\mathbb{G})$ corresponds to a positive
definite function $\varphi$ on $\Gamma$ with $\varphi(e)=1$ via the relation $\Phi(\lambda(\gamma))=\varphi(\gamma)$
for all $\gamma\in\Gamma$. Therefore we have
\[
\Phi\star\lambda(\gamma)=\lambda(\gamma)\star\Phi=(\Phi\otimes\iota)\Delta(\lambda(\gamma))=(\Phi\otimes\iota)(\lambda(\gamma)\otimes\lambda(\gamma))=\varphi(\gamma)\lambda(\gamma),
\]
so the convolution operators associated to $\Phi$ are just the Fourier-Schur
multiplier on $\Gamma$ associated to $\varphi$. Our preceding argument
in particular yields the following result extending
\cite[Theorem 2(a)]{ritter1984convolution}.
\begin{cor}
\label{cor:cocommutative case_schur}Let $\Gamma$ be a finite group
and $\varphi$ be a positive definite function on $\Gamma$ with $\varphi(e)=1$. Let $M_{\varphi}$
be the associated Fourier-Schur multiplier operator determined by
$M_{\varphi}(\lambda(\gamma))=\varphi(\gamma)\lambda(\gamma)$ for
all $\gamma\in\Gamma$. Then there exists $1\leq p<2$ such that
\[
\|M_{\varphi}x\|_{2}\leq\|x\|_{p},\quad x\in C^{*}(\Gamma)
\]
if and only if $|\varphi(\gamma)|<1$ for any $\gamma\in\Gamma\setminus\{e\}$.\end{cor}
\begin{rem}
\label{rem:finiteness is crucial}We remark that the finite-dimensional
condition cannot be removed in any of our results, including Theorem
\ref{Lp improvement_spectral gap}, Theorem \ref{thm:multiplier_quantum},
Corollary \ref{cor:convolution_quantul}-\ref{cor:cocommutative case_schur}.
Here we give a counterexample illustrating this. Let $\mathbb{T}$
be the unit circle in the complex plane, then $\mathbb{T}$ gives
an infinite compact quantum group. Define an operator $T:C(\mathbb{T})\to C(\mathbb{T})$
by 
\[
T(f)=(1-\lambda)\tau(f)+\lambda f,\quad f\in C(\mathbb{T}),
\]
where $0<\lambda<1$ and $\tau$ denote the usual integral against
the normalized Lebesgue measure on $\mathbb{T}$. Then $T$ is obviously
unital completely positive since so are $\tau$ and the identity map.
It is a left multiplier satisfying $T(f)\,\hat{}\,(0)=\hat{f}(0)$
and $T(f)\,\hat{}\,(n)=\lambda\hat{f}(n)$ for $n\neq0$. Also we
may view $T$ as a convolution operator associated to the state $f\mapsto(1-\lambda)\tau(f)+\lambda f(1)$
on $C(\mathbb{T})$, which is faithful since $\tau$ is faithful.
Note that $T$ admits the spectral gap inequality \eqref{eq:spetral gap}
as well, and in fact, $\|Tf\|_{2}=\lambda\|f\|_{2}<\|f\|_{2}$ for
any $f\in C(\mathbb{T})$ with $\tau(f)=0$. However, there doesn't
exist any $p<2$ such that $\|Tf\|_{2}\leq\|f\|_{p}$ for all $f\in C(\mathbb{T})$.
Indeed if such a $p$ existed, then for any $f\in C(\mathbb{T})$,
we would have 
\begin{align*}
\|f\|_{2}^{2} & \geq\|f\|_{p}^{2}\geq\|Tf\|_{2}^{2}=\tau(f)^{2}+\lambda^{2}\|f-\tau(f)\|_{2}^{2}\\
 & \geq\lambda^{2}(\tau(f)^{2}+\|f-\tau(f)\|_{2}^{2})=\lambda^{2}\|f\|_{2}^{2},
\end{align*}
which yields an impossible equivalence between the norms $\|\cdot\|_{2}$
and $\|\cdot\|_{p}$. 
\end{rem}
In spite of the above general remark, Theorem \ref{thm:free prod Lp improving}
still gives constructions of $L_{p}$-improving positive convolution
operators on infinite compact quantum groups. Let $\mathbb{G}_{1},\ldots,\mathbb{G}_{n}$
be finite quantum groups with Haar states $h_1,\ldots,h_n$ respectively and let each $\varphi_{i}$ be a state on
$C(\mathbb{G}_{i})$, $i\in\{1,\ldots,n\}$. Denote $\mathbb{G}=\mathbb{G}_{1}\hat{*}\cdots\hat{*}\,\mathbb{G}_{n}$
with the Haar state $h$ and consider the convolution operators $T_{i}:x\mapsto x\star\varphi_{i},\, x\in C(\mathbb{G}_{i})$.
Note that the free product map $T=*_{1\leq i\leq n}T_{i}$ on $C(\mathbb{G})$ is just the convolution operator given by the c-free product state $\varphi=*_{(h_1,\ldots,h_n)}\varphi_{i}$, i.e.,
\begin{equation}
T(x)=(\varphi\otimes\iota)\Delta(x),\quad x\in C(\mathbb{G}).\label{eq:free product of convolution}
\end{equation}
In fact, we note that if $h(a)=0$, then $h(a_{(1)})=h(a_{(2)})=0$
by \eqref{eq:comultiplication}, where $\Delta(a)\coloneqq\sum a_{(1)}\otimes a_{(2)}$
denotes the Sweedler notation. Now for a reduced word $x=x^{1}\cdots x^{m}$
with $x^{k}\in C(\mathbb{G}_{i_{k}})$ such that $h(x^{k})=0$, $i_{1}\neq\cdots\neq i_{n}$,
$i_{k}\in\{1,\ldots,n\}$ for each $k=1,\ldots,m$,, we have
\begin{align*}
T(x) & =T_{i_{1}}(x^{1})\cdots T_{i_{m}}(x^{m})=\sum\varphi_{i_{1}}(x_{(1)}^{1})x_{(2)}^{1}\cdots\sum\varphi_{i_{m}}(x_{(1)}^{m})x_{(2)}^{m}\\
 & =\sum\varphi(x_{(1)})x_{(2)}=(\varphi\otimes\iota)\Delta(x)
\end{align*}
where we have used the fact that the comultiplication $\Delta$ is
an homomorphism. Then the equality \eqref{eq:free product of convolution}
follows from a standard density argument. Now taking in Theorem \ref{thm:free prod Lp improving}
each $T_{i}$ to be a convolution operator on a finite quantum group,
we get the following corollary:
\begin{cor}
\label{cor:free prod conv}Let $\mathbb{G}_{1},\ldots,\mathbb{G}_{n}$
be finite quantum groups and let each $\varphi_{i}$ be a state on
$C(\mathbb{G}_{i})$, $i\in\{1,\ldots,n\}$. Denote $\mathbb{G}=\mathbb{G}_{1}\hat{*}\cdots\hat{*}\,\mathbb{G}_{n}$
and $\varphi=*_{(h_1,\ldots,h_n)}\varphi_{i}$. If each $\varphi_{i}$
satisfies any one of the conditions \emph{(1)-(5)} in Corollary \ref{cor:convolution_quantul},
then the free product convolution operator given by $T:x\mapsto x\star\varphi,\, x\in C(\mathbb{G})$
is a unital left multiplier on $\mathbb{G}$ satisfies 
\[
\|T:L_{p}(\mathbb{G})\to L_{2}(\mathbb{G})\|=1
\]
for a certain $1<p<2$.\end{cor}
\begin{example}
Now we give a simple method to create nontrivial $L_{p}$-improving positive
convolutions (i.e. the associated state is different from the Haar state) on finite and infinite compact quantum groups. Let $\mathbb{G}$
be a finite quantum group and $h$ the Haar state on it. Given any
state $\varphi$ on $C(\mathbb{G})$ and any $0<\lambda<1$, we can
define a state $\rho$ on $C(\mathbb G)$ by
\[
\rho=\lambda\varphi+(1-\lambda)h.
\]
This is a faithful state which is in particular non-degenerate, 
and hence by Theorem \ref{cor:convolution_quantul} the convolution operator $T_{\rho}:x\mapsto x\star\rho,\, x\in C(\mathbb{G})$
satisfies 
\[
\|T_{\rho}:L_{p}(\mathbb{G})\to L_{2}(\mathbb{G})\|=1
\]
for a certain $1<p<2$ according to Theorem \ref{cor:convolution_quantul}.
Moreover by Corollary \ref{cor:free prod conv}, the convolution operator
$T_{\rho'}:x\mapsto x\star\rho',\, x\in C(\mathbb{G}\hat{*}\mathbb{G})$
given by the c-free product state $\rho'=\rho*_{(h,h)}\rho$ satisfies 
\[
\|T_{\rho'}:L_{p'}(\mathbb{G}\hat{*}\mathbb{G})\to L_{2}(\mathbb{G}\hat{*}\mathbb{G})\|=1
\]
for some $1<p'<2$.
\end{example}

\subsection*{Acknowledgment}
 The author is indebted to his advisors Quanhua Xu and
Adam Skalski for their helpful discussions and constant encouragement, and to
Professor Gilles Pisier for his careful reading and pointing out a mistake in the
preprint version. The author also thanks the referee for a careful reading of the
manuscript and useful suggestions. This research was partially supported by the
NCN (National Centre of Science), grant no. 2014/14/E/ST1/00525.


\begin{thebibliography}{MVD98}

\bibitem[BB10]{banicabichon2010hopfimage}
T.~Banica and J.~Bichon.
\newblock Hopf images and inner faithful representations.
\newblock {\em Glasg. Math. J.}, 52(3):677--703, 2010.

\bibitem[BFS12]{banicafranzskalski2012inner}
T.~Banica, U.~Franz, and A.~Skalski.
\newblock Idempotent states and the inner linearity property.
\newblock {\em Bull. Pol. Acad. Sci. Math.}, 60(2):123--132, 2012.

\bibitem[BLS96]{bozejkoleinertspeicher96condfree}
M.~Bo{\.z}ejko, M.~Leinert, and R.~Speicher.
\newblock Convolution and limit theorems for conditionally free random
variables.
\newblock {\em Pacific J. Math.}, 175(2):357--388, 1996.

\bibitem[Cas13]{caspers2013fourier}
M.~Caspers.
\newblock The {$L^p$}-{F}ourier transform on locally compact quantum groups.
\newblock {\em J. Operator Theory}, 69(1):161--193, 2013.

\bibitem[Coo10]{cooney2010hy}
T.~Cooney.
\newblock A {H}ausdorff-{Y}oung inequality for locally compact quantum groups.
\newblock {\em Internat. J. Math.}, 21(12):1619--1632, 2010.

\bibitem[Daw12]{daws2012cpmultiplier}
M.~Daws.
\newblock Completely positive multipliers of quantum groups.
\newblock {\em Internat. J. Math.}, 23(12):1250132, 23, 2012.

\bibitem[Fol95]{folland1995harmonic}
G.~B. Folland.
\newblock {\em A course in abstract harmonic analysis}.
\newblock Studies in Advanced Mathematics. CRC Press, Boca Raton, FL, 1995.

\bibitem[HR70]{hewittross1970abstract}
E.~Hewitt and K.~A. Ross.
\newblock {\em Abstract harmonic analysis. {V}ol. {II}: {S}tructure and
  analysis for compact groups. {A}nalysis on locally compact {A}belian groups}.
\newblock Die Grundlehren der mathematischen Wissenschaften, Band 152.
  Springer-Verlag, New York-Berlin, 1970.

\bibitem[JNR09]{jungeruan2009repquantumgroup}
M.~Junge, M.~Neufang, and Z.-J. Ruan.
\newblock A representation theorem for locally compact quantum groups.
\newblock {\em Internat. J. Math.}, 20(3):377--400, 2009.

\bibitem[Kah10]{kahng2010fourier}
B.-J. Kahng.
\newblock Fourier transform on locally compact quantum groups.
\newblock {\em J. Operator Theory}, 64(1):69--87, 2010.

\bibitem[MVD98]{maesvandaele1998note}
A.~Maes and A.~Van~Daele.
\newblock Notes on compact quantum groups.
\newblock {\em Nieuw Arch. Wisk. (4)}, 16(1-2):73--112, 1998.

\bibitem[NS06]{nicaspeicher2006freeproba}
A.~Nica and R.~Speicher.
\newblock {\em Lectures on the combinatorics of free probability}, volume 335
  of {\em London Mathematical Society Lecture Note Series}.
\newblock Cambridge University Press, Cambridge, 2006.

\bibitem[Obe82]{oberlin1982convolution}
D.~M. Oberlin.
\newblock A convolution property of the {C}antor-{L}ebesgue measure.
\newblock {\em Colloq. Math.}, 47(1):113--117, 1982.

\bibitem[PX03]{pisierxu2003nclp}
G.~Pisier and Q.~Xu.
\newblock Non-commutative {$L\sp p$}-spaces.
\newblock In {\em Handbook of the geometry of {B}anach spaces, {V}ol.\ 2},
  pages 1459--1517. North-Holland, Amsterdam, 2003.

\bibitem[RX16]{ricardxu2014convexLp}
{\'E}.~{Ricard} and Q.~{Xu}.
\newblock {A noncommutative martingale convexity inequality}.
\newblock {\em Ann. Probab.}, 44(2):867-882 (2016).

\bibitem[Rit84]{ritter1984convolution}
D.~L. Ritter.
\newblock A convolution theorem for probability measures on finite groups.
\newblock {\em Illinois J. Math.}, 28(3):472--479, 1984.

\bibitem[SS15]{skalskisoltan2014quantumfamily}
A.~Skalski and P.~M. So{\l}tan.
\newblock Quantum families of invertible maps and related problems.
\newblock {\em Canad. J. Math.}, 68(3):698-720, 2016.

\bibitem[So{\l}05]{soltan2005bohr}
P.~M. So{\l}tan.
\newblock Quantum {B}ohr compactification.
\newblock {\em Illinois J. Math.}, 49(4):1245--1270, 2005.

\bibitem[Tak02]{takesaki2002opeI}
M.~Takesaki.
\newblock {\em Theory of operator algebras. {I}}, volume 124 of {\em
  Encyclopaedia of Mathematical Sciences}.
\newblock Springer-Verlag, Berlin, 2002.
\newblock Reprint of the first (1979) edition, Operator Algebras and
  Non-commutative Geometry, 5.

\bibitem[VD96]{vandaele1996discrete}
A.~Van~Daele.
\newblock Discrete quantum groups.
\newblock {\em J. Algebra}, 180(2):431--444, 1996.

\bibitem[VD97]{vandaele1997haarfinite}
A.~Van~Daele.
\newblock The {H}aar measure on finite quantum groups.
\newblock {\em Proc. Amer. Math. Soc.}, 125(12):3489--3500, 1997.

\bibitem[VD07]{vandaele2007fourier}
A.~Van~Daele.
\newblock The {F}ourier transform in quantum group theory.
\newblock In {\em New techniques in {H}opf algebras and graded ring theory},
  pages 187--196. K. Vlaam. Acad. Belgie Wet. Kunsten (KVAB), Brussels, 2007.

\bibitem[VDN92]{vdn1992freerandom}
D.-V. Voiculescu, K.~J. Dykema, and A.~Nica.
\newblock {\em Free random variables}, volume~1 of {\em CRM Monograph Series}.
\newblock American Mathematical Society, Providence, RI, 1992.
\newblock A noncommutative probability approach to free products with
  applications to random matrices, operator algebras and harmonic analysis on
  free groups.

\bibitem[Wan95]{wang1995freeprod}
S.~Wang.
\newblock Free products of compact quantum groups.
\newblock {\em Comm. Math. Phys.}, 167(3):671--692, 1995.

\bibitem[Wor98]{woronowicz1998note}
S.~L. Woronowicz.
\newblock Compact quantum groups.
\newblock In {\em Sym\'etries quantiques ({L}es {H}ouches, 1995)}, pages
  845--884. North-Holland, Amsterdam, 1998.

\end{thebibliography}
\end{document}